\documentclass[reqno, 12pts]{amsart}
\usepackage{amsmath,amssymb,amsthm,amsfonts,amsrefs,graphicx,color}
\usepackage[colorlinks=true]{hyperref}
\hypersetup{urlcolor=blue, citecolor=blue}

\usepackage{textcomp}
\usepackage{yfonts, esint}
\usepackage{amssymb}
\usepackage[hmargin=3cm, vmargin=3cm]{geometry}
\usepackage{hyperref}

\newcommand{\smallO}[1]{\ensuremath{\mathop{}\mathopen{}o\mathopen{}\left(#1\right)}}

\def \D{\Delta}
\def\R{\mathbb R}% tap so thuc
\def\le{\leq}% nho hon
\def\ge{\geq}% lon hon
\def\i0i{\int_0^\infty}

\numberwithin{equation}{section}
\newtheorem{thm}{Theorem}%[section]
\newtheorem{lem}{Lemma}
\newtheorem{definition}{Definition}
\newtheorem{rem}{Remark}

\allowdisplaybreaks

\let\ssection=\section\renewcommand{\section}{\setcounter{equation}{0}\ssection}
\newcommand{\abs}[1]{\left |#1\right |}

\def\angb<#1>{\langle #1 \rangle}%% angle bracket
\def\gs{\sigma}

\numberwithin{equation}{section}

\title[Asymptotically homogeneous solutions of the supercritical Lane-Emden system]{Asymptotically homogeneous solutions of the supercritical Lane-Emden system}
\author{Louis Dupaigne}
\address{Universite Claude Bernard Lyon 1, ICJ  UMR5208, CNRS, Ecole Centrale de Lyon, INSA Lyon, Universite Jean Monnet, 69622 Villeurbanne, France}
\email{louis.dupaigne@math.cnrs.fr}

\author{Marius Ghergu}
\address{School of Mathematics and Statistics, University College Dublin, Belfield Campus, Dublin 4, Ireland}
\address{Institute of Mathematics of the Romanian Academy, 21 Calea Grivitei St.,  010702 Bucharest, Romania}
\email{marius.ghergu@ucd.ie}
\author{Hatem Hajlaoui}
\address{ Higher Institute of Applied Mathematics and Computer Sciences of Kairouan, Tunisia}
\email{hajlaouihatem@gmail.com}

\begin{document}

\maketitle
\begin{abstract}
We consider the Lane-Emden system
\begin{align*}%\label{0.1}
-\Delta u =  \vert v\vert^{p-1}v, \quad-\Delta v= \vert u\vert^{q-1}u\qquad\mbox{in }\; \mathbb{R}^d.
\end{align*}
When $p\geq q\geq 1,$ it is known that there exists a positive radial stable solution $(u,v)\in C^2(\R^d)$ if and only if $d\ge11$ and $(p,q)$ lies on or above the so-called Joseph-Lundgren curve introduced in \cite{cdg}.
In this paper, we prove that for $d\leq 10,$ there is no positive stable solution (or merely stable outside a compact set and $(p,q)$ does not lie on the critical Sobolev hyperbola), while for $d\geq 11,$  the Joseph-Lundgren curve is indeed the dividing line for the existence of such solutions, if one assumes in addition that they are asymptotically homogeneous (see Definition \ref{def1} below). Most of our results are optimal improvements of previous works in the litterature.
\end{abstract}

\section{Introduction}
\setcounter{equation}{0}
The Lane-Emden system
\begin{align}\label{1.1}
-\Delta u =  \vert v\vert^{p-1}v, \quad-\Delta v= \vert u\vert^{q-1}u\qquad\mbox{in }\; \mathbb{R}^d,
\end{align}
where $d\ge2$, $p\ge q>0$, $pq>1$ and $u,v\in C^2(\R^d)$, has been studied for the past three decades. Yet, not so much is known about its solutions. Thanks to the works of Mitidieri \cite{mitidieri} and Serrin and Zou \cite{sz}, there exists a radial positive solution to the system if and only if the exponents are supercritical i.e.
\begin{equation}\label{supercritical}
\frac1{p+1}+\frac1{q+1}\le 1-\frac2d.
\end{equation}
The Lane-Emden conjecture states that, whether radial or not, no positive solution exists in the subcritical case i.e. when \eqref{supercritical} does not hold. The current best known result is due to Souplet \cite{souplet}, who proved the conjecture for $d\le 4$, while only partial results are available for $d\ge5$.

We address here the supercritical case and focus on solutions which are stable outside a compact set, i.e. such that there exists a compact set $K\subset\R^d$ and two positive functions $\phi,\psi\in C^2(\R^d\setminus K)$ such that
$$
 -\Delta \phi \ge p\vert v \vert^{p-1}\psi, \quad -\Delta \psi \ge q\vert u \vert^{p-1}\phi\qquad\mbox{in }\; \mathbb{R}^d\setminus K.
$$
We will explain in a short moment why such an assumption is natural.
Thanks to the scaling invariance of the equation, there exists a singular solution of the form
\begin{equation}\label{singu}
(u_s,v_s)(x)=(a|x|^{-\alpha}, b|x|^{-\beta}), \quad x\in \R^d\setminus \{0\}
\end{equation}
where the scaling exponents are given by $$\alpha=\frac{2(p+1)}{pq-1}, \;\;\beta=\frac{2(q+1)}{pq-1}$$
and $a,b$ are suitable (explicit) positive constants. In the supercritical regime, thanks to a suitable version of Hardy's inequality, such a solution is stable if and only if %belongs (componentwise) to $H^1_{loc}(\R^d)$ and solves the equation in the variational sense. In addition, thanks to a suitable version of Hardy's inequality, $(u_s,v_s)$ is stable if and only if
$(p,q)$ lies on or above the Joseph-Lundgren curve i.e. when $d\ge11$ and
$
H^2\ge pq\lambda\mu,
$
where
\begin{equation}\label{ghlm}
    \gamma=\alpha-\beta,\quad H=\frac{(d-2)^2-\gamma^2}{4}, \quad \lambda=\alpha(d-2-\alpha)\quad\text{ and }\quad\mu=\beta(d-2-\beta).
\end{equation}
When $p\ge q\ge1$, Chen and the first two named authors proved in \cite{cdg} that there exists a positive radial stable solution $(u,v)\in C^2(\R^d)$ of \eqref{1.1} if and only if $d\ge11$ and $(p,q)$ lies on or above the Joseph-Lundgren curve. In the case $d\le 10$, we obtain the following optimal improvement.
\begin{thm}\label{tbth3}
Let $d\le 10$ and $p\ge q\ge1$ be such that $pq>1$.
If $(u,v)\in C^2(\R^d)$ is a nonnegative solution which is stable, or merely stable outside a compact set but with
\begin{equation}\label{not critical}
    \frac1{p+1}+\frac1{q+1}\neq 1-\frac2d,
\end{equation}
then $u=v=0$.
\end{thm}
In dimension $d\ge11$, our results are sharp for a restricted class of solutions as we describe next.
\begin{definition}\label{def1}
If it exists, a blow-down limit of a solution $(u,v)\in C^2(\R^d)$ is a cluster point $(u_\infty,v_\infty)$ for the topology of uniform convergence on compact sets of $\R^d\setminus\{0\}$  of the family $(u_R,v_R)_{R\ge1}$ of rescalings defined by
\begin{equation}\label{rescalings}
(u_R,v_R)(x)=(R^{\alpha}u(Rx), R^{\beta}v(Rx))\quad\text{ for $x\in\R^d$}.
\end{equation}
A solution $(u,v)$ is said to be asymptotically homogeneous if all its blow-down limits are homogeneous i.e. if  there exists $(f,g)\in C^2(S^{d-1})$ such that $(u_\infty,v_\infty)(r\theta)=(r^{-\alpha }f(\theta), r^{-\beta} g(\theta))$ for $r>0$ and $\theta\in S^{d-1}$.
\end{definition}
With this definition in mind, we obtain the following theorem.
\begin{thm}\label{tbth2}
Let $p\ge q\ge1$ be such that $pq>1$ and
\begin{equation}\label{1.7}
  H^2<pq\lambda\mu,
\end{equation}
where $H,\lambda,\mu$ are given by \eqref{ghlm}.
If $(u,v)\in C^2(\R^d)$ is a nonnegative solution which is stable (resp. stable outside a compact set and  \eqref{not critical} holds) and asymptotically homogeneous, then $u=v=0$.
\end{thm}
Blow-down limits of nonnegative solutions of the Lane-Emden system which are stable outside a compact set are always well-defined:
\begin{thm}\label{prop-gamma0} Let $(u,v)\in C^2(\R^d)$ be a nonnegative solution of \eqref{1.1} which is stable outside a compact set. Then,
 the family $(u_R,v_R)_{R\ge1}$  given by \eqref{rescalings} is %bounded in $H^1_{loc}(\R^d\setminus\{0\})$ and
 compact in $L^{q+1}_{loc}(\R^d\setminus\{0\})\times L^{p+1}_{loc}(\R^d\setminus\{0\})$.
\end{thm}
In addition, in the cases $p=q$ and $p> q=1$, all blow-down limits are homogeneous, thanks to the availability of a monotonicity formula (see \cite{ddww}). Unfortunately, we do not know if this fact is also true for the Lane-Emden system for other choices of exponents. Still, we can prove that our results continue to hold for a possibly wider class of solutions.
\begin{thm}\label{tbth}
Let $p\ge q\ge1$ be such that $pq>1$ and
\eqref{1.7} holds. Let $K\subset\R^d$ be a compact set.
If $(u,v)\in C^2(\R^d)$ is a nonnegative solution of \eqref{1.1} such that for all $\varphi\in C^1_c(\R^d)$ (resp. for all $\varphi\in C^1_c(\R^d\setminus K)$ and \eqref{not critical} holds),
\begin{equation}\label{stability interp improved}
\int_{\R^d}\vert u\vert^{\frac{q-1}2}\vert v\vert^{\frac{p-1}2}\varphi^2\;dx \le \frac1{\sqrt{pq}}\left(\int_{\R^d}\vert\nabla\varphi\vert^2\;dx-\frac{\gamma^2}{4}\int_{\R^d}\frac{\varphi^2}{\vert x\vert^2}\;dx\right)
\end{equation}
then $u=v=0$.
\end{thm}
Assumption \eqref{stability interp improved} is motivated by the following observation.
\begin{thm}\label{prop-gamma} Let $(u,v)\in C^2(\R^d)$ be a nonnegative solution of \eqref{1.1} which is stable outside a compact set and asymptotically homogeneous. Let $(u_\infty,v_\infty)$ denote a blow-down limit. Then, $(u_\infty,v_\infty)$ satisfies \eqref{stability interp improved} for any $\varphi\in C^1_c(\R^d)$.
\end{thm}
Finally, we obtain the following partial result for solutions which need not be asymptotically homogeneous.
\begin{thm}\label{tbth4}
Let $p\ge q\ge1$ be such that $pq>1$ and
\begin{equation}\label{tech}
d< 2+2x_0,
\end{equation}\label{polyH}
where $x_0$ is the largest root of the polynomial
\begin{equation}\label{}
  H(x)=x^4-pq\alpha\beta(4x^2-2(\alpha+\beta)x+\alpha\beta).
\end{equation}
If $(u,v)\in C^2(\R^d)$ is a nonnegative solution of \eqref{1.1} which is stable (resp. stable outside a compact set and \eqref{not critical} holds), then $u=v=0$.
\end{thm}
\begin{rem} The sharp condition \eqref{1.7} can be reformulated as \eqref{tech} where this time $x_0$ must be interpreted as the largest root of the polynomial
$$
H_{JL}(x)=(x^2-\gamma^2/4)^2-pq\alpha\beta(4x^2-2(\alpha+\beta)x+\alpha\beta).
$$
\end{rem}

 %In fact, the positivity assumption can be removed in that case:
%\begin{thm}\label{thm-homo}Let $p\ge q>0$ be such that $pq>1$, \eqref{not critical} and \eqref{1.7} hold. If $(u,v)\in H^1_{loc}(\R^d)$ is a homogeneous stable solution of \eqref{1.1}, then $u=v=0$.
%\end{thm}
The above theorems %is sharp and
build upon a wealth of previously known results, as we describe next. Thanks to the comparison inequality
\begin{equation}\label{souplet ineq}
\frac{v^{p+1}}{p+1}\le\frac{u^{q+1}}{q+1}
\end{equation}
established for $p\ge q$ and for bounded positive solutions by Souplet in \cite{souplet}, such solutions of \eqref{1.1} must satisfy $u=v$ in the case $p=q$, so that the system becomes a single equation and the condition \eqref{1.7} reduces to $p<p_c(d)$, where $p_c(d)$ is the Joseph-Lundgren stability exponent discovered in \cite{jl}. Crandall and Rabinowitz proved in \cite{cr} that, given a smoothly bounded domain $\Omega\subset\R^d$ and the nonlinearity $f(u)=(1+u)^p$, the equation
\begin{equation*}
    \left\{
    \begin{aligned}
    -\Delta u&=\lambda f(u)&\quad\text{in $\Omega$,}\\
    u&=0&\quad\text{on $\partial\Omega$,}
    \end{aligned}
    \right.
\end{equation*}
admits a curve of solutions $\lambda\in [0,\lambda^*)\mapsto u_\lambda\in C^2(\overline\Omega)$, which are stable. They applied Moser's method \cite{moser} with a twist: substituting Sobolev's inequality by the variational formulation of the stability property in the iteration process, they proved uniform boundedness of the family $(u_\lambda)_{\lambda\in[0,\lambda^*)}$ in $L^\infty(\Omega)$ for $p<p_c(d)$. This implies the existence of a smooth stable solution $u^*$ also for the extremal parameter $\lambda=\lambda^*$. Conversely, for $p\ge p_c(d)$ and $\Omega=B_1$, the extremal solution $u^*$ still exists but is singular (and $u^*-1$ is given by \eqref{singu} for $p=q$).

For $p<p_S(d)$ where $p_S(d)$ is the Sobolev critical exponent and $f(u)=\vert u\vert^{p-1}u$, Bahri and Lions \cite{bl} proved that the $L^\infty$ bound extends and is in fact equivalent to a bound on the Morse index of (possibly sign-changing) solutions. They argued by a blow-up argument and proved Theorem \ref{tbth} for finite Morse index solutions of the subcritical equation. Farina extended their result to all $p<p_c(d)$, $p\neq p_S(d)$  and to solutions stable outside a compact set, observing that finite Morse index solutions belong to this class. In fact the two notions coincide, as proved by Devyver \cite{devyver}. In other words, Theorem \ref{tbth2} generalizes Farina's result to the case of positive asymptotically homogeneous solutions of the system \eqref{1.1}.
We note that the restriction $p\neq p_S(d)$ is necessary in Farina's result, and so is the restriction \eqref{not critical} for the Lane-Emden system \eqref{1.1}. Indeed, in the critical case, Lions proved in \cite{lions} that \eqref{1.1} has a ground state solution, which is unique up to scaling and translation. He observed that if $(u,v)$ is a positive solution of \eqref{1.1}, then
$$
-\Delta((-\Delta v)^{\frac1q})=v^p\quad\text{in $\R^d$}
$$
and so a ground state solution can be sought by minimizing $\mathcal E(v)=\Vert \Delta v\Vert_{L^m(\R^d)}$, $m=\frac{q+1}{q}$, over the set of functions $v\in\mathcal D^{2,m}(\R^d)$ such that $\Vert v\Vert_{L^{p+1}(\R^d)}=1$. Such a solution is positive, radial and stable outside a compact set by construction (or as follows from the asymptotics computed by Hulshof and Van der Vorst in \cite{hv}). As observed by Mtiri and Ye in \cite{MY}, the second variation of the energy $\mathcal E$ can be computed at positive stable solutions in the case $p> 1\ge q$, leading to the following reformulation of stability:
\begin{equation}\label{stability}
\int_{\R^d} \vert v \vert^{p-1}\varphi^2\;dx \le \frac1{pq}\int_{\R^d} u^{1-q}\vert\Delta\varphi\vert^2\;dx\quad\text{for all $\varphi\in C^2_c(\R^d)$}.
\end{equation}
%As we shall prove, this inequality actually holds for any solution of \eqref{1.1} stable outside a compact set and without restriction on the exponents, with the convention that the right-hand side equals $+\infty$ if the intersection of the support of $\varphi$ with the set $[u=0]$ has positive Lebesgue measure.

Let us turn now to the well understood biharmonic case $q=1$. In that case and for positive radial stable solutions, Theorem \ref{tbth} follows from the works of Gazzola, Grunau \cite{gg}, Guo, Wei \cite{gw} and Karageorgis \cite{k}. Since the Moser iteration method is based on the chain rule, its adaptation to fourth order equations is nontrivial. Still, Wei and Ye \cite{wy} classified positive stable solutions when $d\le 8$ thanks to the inequality \eqref{souplet ineq} (a strategy inspired by earlier work of Cowan, Esposito and Ghoussoub \cite{ceg}).  Following the same Moser iteration strategy, Harrabi, Ye and the third named author classified positive stable solutions for $d\le 12$ in \cite{hhy}. They exploited the following interpolated version of the stability inequality:
\begin{equation}\label{stability interp}
\int_{\R^d}\vert u\vert^{\frac{q-1}2}\vert v\vert^{\frac{p-1}2}\varphi^2\;dx \le \frac1{\sqrt{pq}}\int_{\R^d}\vert\nabla\varphi\vert^2\;dx%\quad\text{for all $\varphi\in C^1_c(\R^d\setminus K)$},
\end{equation}
of which several proofs are available (based on Picone's identity in Cowan \cite{co}, Cowan and Ghoussoub \cite{cg} and Farina, Sirakov and the first author \cite{dfs}, interpolation theory in Goubet, Warnault and the first two named authors \cite{dggw}). Finally, Theorem \ref{tbth} was proved for the full range of exponents in the biharmonic case by D\'avila, Wang, Wei and the first named author \cite{ddww}. Their result crucially relies on a monotonicity formula which is unavailable for the system.

The complete system was studied by Cowan \cite{co}, who classified positive stable solutions for $d\le 10$ and $p\ge q\ge2$ thanks to \eqref{stability interp}. His result was improved by Harrabi, Mtiri and the third named author \cite{HHM}, who classified positive stable solutions for $d\le 10$ and $p\ge q>4/3$ and bounded positive stable solutions for $d\le 6$ and $p\ge q>1$. Also, Mtiri and Ye \cite{MY} completely classified positive solutions stable outside a compact set for $(p,q)$ subcritical.

In order to prove Theorem \ref{tbth}, we propose a new strategy, where (near optimizers of) Hardy's inequality turn out to have a central role. Due to the lack of a monotonicity formula, any blow-down (resp. blow-up) limit of a given solution need not be homogeneous {\it a priori}. %Nevertheless, we prove that such a limit satisfies an improved version of the estimate \eqref{stability interp}, namely
To bypass this difficulty, we assume that \eqref{stability interp improved} holds. Now, rather than Moser iteration, we feed \eqref{stability interp improved} in an iteration scheme which is closer to De Giorgi's original idea \cite{dg}. More precisely, we prove a reduction of the oscillation lemma for solutions satisfying \eqref{stability interp improved}. We capture the oscillations through the rescaled and renormalized Dirichlet energy
\begin{equation}\label{dirichlet}
    r^2\fint_{B_r}\left\vert\nabla \left(\frac{u}{u_s}\right)^{a} \right\vert^2dx + r^2\fint_{B_r}\left\vert\nabla \left(\frac{v}{v_s}\right)^{b} \right\vert^2dx
\end{equation}
where $(a,b)=\left(\frac{q+1}{2}, \frac{p+1}{2}\right)$ and $(u_s,v_s)$ is given by \eqref{singu}. Thanks to Campanato's characterization of H\"older spaces \cite{campanato}, this is enough to obtain the following universal H\"older estimate on local stable solutions of the system:
\begin{equation}\label{elliptic regularity}
    \left\Vert{u}/{u_s}\right\Vert_{C^\sigma(B_1)}+ \left\Vert{v}/{v_s}\right\Vert_{C^\sigma(B_1)} \le C%\left(\Vert u\Vert_{L^1(B_2)}+\Vert v\Vert_{L^1(B_2)}\right)
\end{equation}
which holds for some $\sigma=\sigma(d,p,q)\in (0,1)$ and $C=C(d,p,q)>0$ under the assumption \eqref{stability interp improved}.
Note that thanks to the translation invariance of the Lane-Emden system, the above estimate can be applied to any translation of a given solution and so H\"older regularity for the solution itself follows. Also, by a natural scaling argument, the classification follows for all solutions such that \eqref{stability interp improved} holds in $\R^d$. In particular, blow-down limits are trivial under the assumptions of Theorem \ref{tbth}. This gives enough asymptotic information to conclude that solutions such that \eqref{stability interp improved} holds only outside a compact set are also trivial. Estimate \eqref{elliptic regularity} bears resemblance with the work of Cabr\'e, Figalli, Ros-Oton and Serra \cite{cfrs}. In their case, a version of Pohozaev's identity is central to the analysis. In our framework, we derive \eqref{elliptic regularity} through a much different road, where Hardy's inequality is used instead. It is also interesting to note that the scale invariance of the equation (resp. the nonlinear nature of the problem) appears explicitly in the definition of the Dirichlet energy \eqref{dirichlet} through the normalization by $(u_s,v_s)$ (resp. through the exponents $a,b$ reminiscent of Moser's iteration).

\section{Stability revisited} In this section, we derive all the functional inequalities presented in the introduction. They are reformulations (or sometimes merely consequences of) stability and they are at the heart of the iteration methods used in this paper. Recall that a solution $(u,v)\in C^2(\Omega)$ is said to be stable in an open set $\Omega\subset\R^d$ if there exist two positive functions $\phi,\psi\in C^2(\Omega)$ such that
\begin{equation}\label{stabb}
 -\Delta \phi \ge pv^{p-1}\psi, \quad -\Delta \psi \ge qu^{q-1}\phi\qquad\mbox{in }\; \Omega.
\end{equation}
In the next two lemmas, we give a variational reformulation of \eqref{stabb}.
\begin{lem}\label{2}Let $\varphi\in C^2(\Omega)$ be a positive superharmonic function and $\eta\in C^2(\Omega)$. Then,
$$
\Delta\left(\frac{\eta^2}\varphi\right)\Delta\varphi \le (\Delta\eta)^2.
$$
\end{lem}
\begin{proof}
Expand $\Delta\left(\frac{\eta^2}\varphi\right)$:
\begin{align*}
\Delta\left(\frac{\eta^2}\varphi\right)&= \frac1\varphi\Delta(\eta^2)+2\nabla\eta^2\cdot\nabla\frac1{\varphi}+\eta^2\Delta\left(\frac1\varphi\right)\\
&=\frac2\varphi(\eta\Delta\eta+\vert\nabla\eta\vert^2) - 4\frac{\eta}{\varphi^2}\nabla\eta\cdot\nabla\varphi + \eta^2\left(-\frac1{\varphi^2}\Delta\varphi+\frac2{\varphi^3}\vert\nabla\varphi\vert^2\right)\\
&=\left[
2\frac\eta\varphi\Delta\eta -\frac{\eta^2}{\varphi^2}\Delta\varphi
\right]+
\left[
\frac2\varphi\vert\nabla\eta\vert^2 - 4\frac{\eta}{\varphi^2}\nabla\eta\cdot\nabla\varphi+2\frac{\eta^2}{\varphi^3}\vert\nabla\varphi\vert^2
\right]\\
&=-\left[
\frac{\eta^2}{\varphi^2}\Delta\varphi-
2\frac\eta\varphi\Delta\eta
\right]+
\frac2{\varphi}\left[
\vert\nabla\eta\vert^2 - 2\frac{\eta}{\varphi}\nabla\eta\cdot\nabla\varphi+\frac{\eta^2}{\varphi^2}\vert\nabla\varphi\vert^2
\right].
\end{align*}
Observe that the second bracket above is a perfect square. Multiply by $\Delta\varphi$ and complete the square in the first bracket to get
\begin{equation}\label{squares}
\Delta\left(\frac{\eta^2}\varphi\right)\Delta\varphi=
\left[
-\left(\frac\eta\varphi\Delta\varphi-\Delta\eta\right)^2+(\Delta\eta)^2
\right]+
2\frac{\Delta\varphi}\varphi
\left\vert
\nabla\eta - \frac\eta\varphi \nabla\varphi
\right\vert^2
\le (\Delta\eta)^2,
\end{equation}
since $\Delta\varphi\le0$.
\end{proof}

\begin{lem}\label{lem2}Assume that $(u,v)\in C^2(\Omega)$ is positive and stable in $\Omega$. Let $\Omega'\Subset\Omega$. Then, for every $\eta\in H^1_0(\Omega')\cap H^2(\Omega')$,% supported away from the zeros of $u$,
$$pq\int v^{p-1}\eta^2 \le \int  u^{1-q}(\Delta \eta)^2$$
and
\begin{equation}\label{stabvar}
pq\int u^{q-1}\eta^2 \le \int  v^{1-p}(\Delta \eta)^2.
\end{equation}

\end{lem}

\begin{proof}By symmetry, it suffices to prove the first inequality. We may also assume that $\eta\in C^2(\overline{\Omega'})$. Multiply the first inequality in \eqref{stabb} by $q\eta^2/\psi$ and integrate over $\Omega'$ to get
$$
pq\int v^{p-1}\eta^2\le q\int(-\Delta\phi)\frac{\eta^2}\psi= q\int -\Delta\left(\frac{\eta^2}\psi\right)\phi\le q\int_{[-\Delta\left(\frac{\eta^2}\psi\right)\ge0]} -\Delta\left(\frac{\eta^2}\psi\right)\phi.
$$
By the second inequality in \eqref{stabb}, it follows that
$$
pq\int v^{p-1}\eta^2\le \int_{[-\Delta\left(\frac{\eta^2}\psi\right)\ge0]}  u^{1-q}\Delta\left(\frac{\eta^2}\psi\right)\Delta\psi,
$$
and the conclusion follows by Lemma \ref{2}.
\end{proof}
In the next lemma, we characterize stability for homogeneous blow-down limits.

\begin{lem}\label{lemma3} Let  $(u,v)\in C^2(\R^d)$ be a solution of \eqref{1.1}. Assume that $(u,v)\in C^2(\R^d)$ is positive, stable outside a compact set and asymptotically homogeneous. Let $(u_\infty,v_\infty)(r\theta)=(r^{-\alpha} f(\theta), r^{-\beta} g(\theta))$ denote a blow-down limit of $u$. Then, for every $\varphi\in C^2(S^{d-1})$,
$$
pq\int_{S^{d-1}} \vert g\vert^{p-1}\varphi^2\le \int_{S^{d-1}}  f^{1-q}\vert\Delta_\theta\varphi-H\varphi\vert^2,
$$
where $H$ is given by \eqref{ghlm}.
\end{lem}

\begin{proof}Fix an open set $\Omega\Subset\R^d\setminus\{0\}$. Since $(u,v)$ is stable outside a compact set $K$, its rescaling $(u_R,v_R)(x)=(R^\alpha u(Rx), R^{\beta} v(Rx))$ is stable outside $K/R$ and so it is stable in an open neighborhood of $\overline\Omega$ for $R$ large enough.
%Let
%$$
%\lambda_1=\inf
%\left\{\int v_R^{1-p}(\Delta \eta)^2-
%pq\int u_R^{q-1}\eta^2\;:\;\eta\in H^1_0(\Omega)\cap H^2(\Omega)\text{ s.t. }\Vert\eta\Vert_{L^2(\Omega)}=1
%\right\}.
%$$
%By Lemma \ref{lem2}, we have $\lambda_1\ge0$. Since $(u_R,v_R)$ is bounded and bounded away from zero in the compact set $\overline\Omega$, $\lambda_1$ is achieved by some function $\phi_R$ weakly solving the linearized equation
%$$
% -\Delta \phi_R = pv_R ^{p-1}\psi_R, \quad -\Delta \psi_R = (qu_R ^{q-1}+p^{-1}\lambda_1)\phi_R\qquad\mbox{in }\; \Omega,
%$$
%where $\psi_R$ is defined through the first equation above.
%By standard elliptic regularity, $(\phi_R,\psi_R)\in C^2(\Omega)$ and the equation is satisfied in the classical sense. Note that the solution $\tilde\phi_R\in H^1_0(\Omega)\cap H^2(\Omega)$ to $-\Delta \tilde\phi_R=\vert \Delta \phi_R\vert$ also solves the minimization problem, so that we can assume $(\phi_R,\psi_R)>0$. In fact, $\lambda_1>0$. If not, inequality \eqref{stabvar} (applied to the solution $(u_R,v_R)$) would be an equality for the choice $\eta=\phi_R$. This would imply in turn that \eqref{squares} is an equality for the choice $(\eta, \varphi)=(\phi_R,\phi)$, leading to $\phi_R=c\phi$ in $\overline\Omega$, for some $c\in\R^*$. The latter equality cannot hold on $\partial\Omega$. Hence, $\lambda_1>0$ and repeating the proof of \eqref{stabvar}, we obtain the slightly finer inequality
By \eqref{stabvar},
%$$
%pq\int v_R^{p-1}\eta^2 \le \int  \left(u_R^{q-1}+\frac{\lambda_1}{pq}\right)^{-1}(\Delta \eta)^2,\qquad\eta\in C^2_c(\Omega)
%$$
$$
pq\int v_R^{p-1}\eta^2 \le \int  u_R^{1-q}(\Delta \eta)^2,\qquad\eta\in C^2_c(\Omega)
$$
%Thanks to Fatou's lemma and the dominated convergence theorem (applied to the family $(u_R,v_R)$ in the limit $R\to+\infty$), we arrive at
Passing to the limit as $R\to+\infty$,
%$$
%pq\int v_\infty^{p-1}\eta^2 \le \int  \left(u_\infty^{q-1}+\frac{\lambda_1}{pq}\right)^{-1}(\Delta \eta)^2\le\int  u_\infty^{1-q}(\Delta \eta)^2.
%$$
$$
pq\int v_\infty^{p-1}\eta^2 \le\int  u_\infty^{1-q}(\Delta \eta)^2.
$$
Since $(u_\infty,v_\infty)$ is homogeneous, we deduce that for $\eta=h(r)\varphi(\theta)$, $h\in C^2_c(0,+\infty)$, $\varphi\in C^2(S^{d-1})$,
\begin{equation}\label{s1}
pq\left(\int_{\R_+}h^2 r^{-\beta(p-1)+(d-1)}dr\right)\left(\int_{S^{d-1}} \vert g\vert^{p-1}\varphi^2\right) \le \int_{\R_+\times S^{d-1}} r^{\alpha(q-1)} f^{1-q}(\Delta \eta)^2.
\end{equation}
We have
$$
\Delta\eta = (\Delta h)\varphi+\frac{h}{r^2}\Delta_\theta\varphi
$$
and so
$$
(\Delta\eta)^2= (\Delta h)^2\varphi^2+ 2\left(\frac{h\Delta h}{r^2}\right)\varphi\Delta_\theta\varphi +\frac{h^2}{r^4}(\Delta_\theta\varphi)^2.
$$
Thus, \eqref{s1} becomes
\begin{align*}
pq\left(\int_{\R_+}h^2 r^{-\beta(p-1)+d-1}dr\right)\left(\int_{S^{d-1}} \vert g\vert^{p-1}\varphi^2\right) \le&
\left(\int_{\R_+}h^2 r^{d-5+\alpha(q-1)}dr\right)
\left(\int_{S^{d-1}} f^{1-q}(\Delta_\theta\varphi)^2
\right)+\\&
\left(\int_{\R_+}2h\Delta h\, r^{d-3+\alpha(q-1)}dr\right)
\left(\int_{S^{d-1}} f^{1-q}\varphi\Delta_\theta\varphi
\right)+\\&
\left(\int_{\R_+}(\Delta h)^2 r^{d-1+\alpha(q-1)}dr\right)
\left(\int_{S^{d-1}} f^{1-q}\varphi^2
\right).
\end{align*}
Recalling the definitions of the scaling exponents $(\alpha,\beta)$, the first integral on the left-hand side is equal to the first integral on the right-hand side in the above inequality. So,
$$
pq\left(\int_{S^{d-1}} \vert g\vert^{p-1}\varphi^2\right) \le
\left(\int_{S^{d-1}} f^{1-q}(\Delta_\theta\varphi)^2
\right)+
C
\left(\int_{S^{d-1}} f^{1-q}\varphi\Delta_\theta\varphi
\right)+
D
\left(\int_{S^{d-1}} f^{1-q}\varphi^2
\right),
$$
where
$$
C=\frac{\int_{\R_+}2h\Delta h\, r^{d-3+\alpha(q-1)}dr}{\int_{\R_+}h^2 r^{d-5+\alpha(q-1)}dr}
\quad\text{
and }
\quad
D=\frac{\int_{\R_+}(\Delta h)^2 r^{d-1+\alpha(q-1)}dr}{\int_{\R_+}h^2 r^{d-5+\alpha(q-1)}dr}.
$$
Choosing $h(r)=r^{-\frac{d-2-\gamma}2}k_n(r)\in C^2_c(\R_+)$, where $k_n$ is a standard cut-off function (vanishing near the origin and near infinity) such that $k_n(r)\to 1$ for every $r\in\R_+^*$, we find that  $C\to -2H$ and $D\to H^2$ as $n\to+\infty$, where $H$ is given by \eqref{ghlm}. Hence,
$$
pq\left(\int_{S^{d-1}} \vert g\vert^{p-1}\varphi^2\right) \le
\int_{S^{d-1}} f^{1-q}\left(
(\Delta_\theta\varphi)^2
-2H\varphi\Delta_\theta\varphi
+
H^2
\varphi^2
\right)=\int_{S^{d-1}} f^{1-q}\left(
\Delta_\theta\varphi
-
H
\varphi
\right)^2.
$$
\end{proof}
In the last two lemmas of this section, we prove that stability implies the weaker but more handy inequality  \eqref{stability interp improved}.
\begin{lem} \label{lemma4} Assume that $(u,v)\in C^2(\R^d)$ is positive and stable outside a compact set. Assume also that $(u,v)$ is asymptotically homogeneous and let $(u_\infty,v_\infty)(r\theta)=(r^{-\alpha} f(\theta), r^{-\beta} g(\theta))$ denote a blow-down limit. Then, for every $\eta\in C^1(S^{d-1})$,
$$
\sqrt{pq}\int_{S^{d-1}}f^{\frac{q-1}2}g^{\frac{p-1}2}\eta^2\le \int_{S^{d-1}}\left(\vert\nabla_\theta\eta\vert^2+H\eta^2\right),
$$
where $H$ is given by \eqref{ghlm}.
\end{lem}

\begin{proof}%Let $\epsilon>0$, $\Omega_\epsilon=[\vert f\vert>\epsilon]\cap[\vert g\vert<\epsilon^{-1}]$ and
Let
$$
\lambda_1=\inf\left\{
\int_{S^{d-1}}  f^{1-q}\vert\Delta_\theta\psi-H\psi\vert^2- pq\int_{S^{d-1}} \vert g\vert^{p-1}\psi^2 \; :\;
\psi\in H^2(S^{d-1})\text{ s.t. $\Vert\psi\Vert_{L^2(S^{d-1})}=1$}
\right\}.
$$
Then, $\lambda_1\ge0$ by Lemma \ref{lemma3}. Since $f$ is bounded away from zero and $g$ is bounded, $\lambda_1$ is attained by some function $\psi\in H^2(S^{d-1})$ such that $\Vert\psi\Vert_{L^2(S^{d-1})}=1$. Replacing $\psi$ by $\tilde\psi$ the weak solution to
$$
-\Delta_\theta \tilde\psi+H\tilde\psi = \vert\Delta_\theta\psi-H\psi\vert\quad\text{in $S^{d-1}$,}
$$
we may assume that $\psi\ge 0$ and $\varphi:=q^{-1} f^{1-q}(-\Delta_\theta \psi+H\psi)\ge 0$. By the strong maximum principle, $(\varphi,\psi)>0$ a.e. Also, $(\varphi,\psi)$ is a weak solution to
$$
 -\Delta_\theta \varphi +H\varphi= (p{\vert g\vert^{p-1}}+\lambda_1)\psi, \quad -\Delta_\theta \psi +H\psi= q{\vert f\vert^{q-1}}\varphi\qquad\mbox{in }\; S^{d-1}
$$
and by elliptic regularity, $(\phi,\psi)\in C^2(S^{d-1})$.
Take $\eta\in C^1(S^{d-1})$, multiply the first equation above by $\eta^2/\varphi$ and integrate by parts. Since $\lambda_1\ge0$,
$$
p\int {\vert g\vert^{p-1}}\frac{\psi}{\varphi}\eta^2 \le \int \nabla_\theta\varphi\cdot\nabla_\theta\left(\frac{\eta^2}{\varphi}\right) + H\int\eta^2 \le \int \left(\vert\nabla_\theta\eta\vert^2+H\eta^2\right).
$$
Similarly,
$$
q\int {\vert f\vert^{q-1}}\frac{\varphi}{\psi}\eta^2 \le\int \left(\vert\nabla_\theta\eta\vert^2+H\eta^2\right).
$$
Multiplying the above two inequalities and applying the Cauchy-Schwarz inequality, the result follows.
\end{proof}

Now, we are ready to provide the proof of Theorem \ref{prop-gamma}.
\begin{proof}[Proof of Theorem \ref{prop-gamma}.] Fix $\varphi\in C_c^1(\R^d)$ and $r>0$. By the previous lemma,
$$
\frac1{\sqrt{pq}} \int_{S^{d-1}}\left(\vert\nabla_\theta\varphi(r\theta)\vert^2+H\varphi(r\theta)^2\right)d\sigma(\theta)\ge\int_{S^{d-1}}f^{\frac{q-1}2}g^{\frac{p-1}2}\varphi(r\theta)^2d\sigma(\theta),
$$
Multiply by $r^{d-3}$ and integrate the above inequality in the $r$-variable. We get
$$
\frac1{\sqrt {pq}}\int_{S^{d-1}\times\R_+} \left(\vert\nabla_\theta\varphi(r\theta)\vert^2+H\varphi(r\theta)^2\right)r^{d-3}drd\sigma(\theta)
\ge
\int_{\R^d}u_\infty^{\frac{q-1}2}v_\infty^{\frac{p-1}2}\varphi^2\,dx.
$$
In addition, by Hardy's inequality, for every $h\in C^1_c(\R_+)$,
$$
\int_{\R_+}\left(\frac{\partial h}{\partial r}\right)^2r^{d-1}dr \ge \left(\frac{d-2}{2}\right)^2\int_{\R_+}\frac{h^2}{r^2}r^{d-1}dr.
$$
Applying Hardy's inequality to $h(r)=\varphi(r\theta)$ for fixed $\theta\in S^{d-1}$, recalling that $\vert \nabla\varphi\vert^2 = \left(\frac{\partial \varphi}{\partial r}\right)^2+\frac1{r^2}\vert\nabla_\theta\varphi\vert^2$ and summing the above two inequalities, \eqref{stability interp improved} follows.
\end{proof}

\section{The case $\alpha+\beta\ge d-4$}\label{sec3}
In this section, we derive integral estimates on solutions thanks to their stability properties. These estimates are the central tool to prove Theorems \ref{tbth3} and \ref{tbth4}. They also imply Theorem \ref{prop-gamma0}.
To begin with, recall the following classical lemma, which holds for all positive solutions of \eqref{1.1}, see  \cites{sz, MPO} and e.g. Lemma 2 in \cite{HHM} for a proof.
\begin{lem}\label{l.2.1}
Let $p\ge q\ge 1$ and let $(u, v)\in C^2(\R^d)$ be a positive solution  of \eqref{1.1}. Then,  there exists a constant $C > 0$ depending on $p,q,d$ only such that for any $R \geq 1$, there holds
\begin{align*}
   \int_{B_{R}}v^{p} dx\leq C R^{d-\beta p},\quad
   \int_{B_{R}}u^{q} dx \leq C R^{d-\alpha q }.
\end{align*}
\end{lem}
We will also use the following comparison inequality, due to Souplet \cite{souplet} for bounded solutions. See e.g. Lemma 2.7 in Q. H. Phan \cite{Pha} for a proof when the solution is not assumed to be bounded.
\begin{lem}\label{Soup} Let $p\ge q\ge 1$ be such that $pq>1$. Then, any %bounded
positive solution $(u,v)\in C^2(\R^d)$ of \eqref{1.1}
verifies
\begin{align}
\label{estS}
v^{p + 1} \leq \frac{p+1}{q+1}u^{q+1} \;\; \mbox{in}\;\;\mathbb{R}^d.
\end{align}
\end{lem}
%\begin{rem} The above lemma holds for bounded solutions. Thanks to the doubling lemma introduced in \cite{PQS}, this assumption is not restrictive and it suffices to prove Theorems \ref{tbth} and \ref{tbth2} for bounded solutions.
%\end{rem}
The next lemma pertains to stable solutions and uses ideas from \cites{hhy, HHM}.
\begin{lem}
\label{Newl}Assume that $(u,v)\in C^2(\Omega)$ is a positive stable solution of \eqref{1.1} in an open set $\Omega\subset\R^d$. Let $a\geq\frac{q+1}{2}$ and $b=\frac{p+1}{q+1}a.$ Assume that $AB>1,$ where $A=\sqrt{pq}\frac{(2a-1)}{a^2}$ and $B=\sqrt{pq}\frac{(2b-1)}{b^2}$. Then, there exists $C >0$ depending on $p,q,d,a$ only such that
\begin{equation*}
\int_{\Omega}v^{p}u^{2a-1}\eta^2dx  \leq C\int_{\Omega}u^{2a}\Big[|\nabla\eta|^2+\vert \Delta(\eta^2) \vert\Big] dx,\quad \forall\, \eta \in C_c^2(\Omega).
\end{equation*}
\end{lem}
\begin{proof} Take $\eta \in C_c^2(\Omega).$  Let $(u,v)$ be a stable solution of
\eqref{1.1} in $\Omega$ and $a>\frac12$. Integrating by parts,
\begin{align}
\label{3.1}
\begin{split}
\int_{\Omega}|\nabla
u^{a}|^2\eta^2 dx & = a^2\int_{\Omega}u^{2a-2}|\nabla u|^2\eta^2 dx\\ & = \frac{a^2}{2a-1}\int_{\Omega}\eta^2\nabla(u^{2a-1})\nabla
u\,dx  \\
&= \frac{a^2}{2a-1}\int_{\Omega}u^{2a-1}v^{p}\eta^2dx  + \frac{a}{2(2a-1)}\int_{\Omega}u^{2a}\Delta(\eta^2)dx ,
\end{split}
\end{align}
and
\begin{align}
\label{3.2}
  2a\int_{\Omega}u^{2a-1}\eta\nabla u\nabla\eta dx  = \frac{1}{2}\int_{\Omega}\nabla(u^{2a})\nabla(\eta^2)dx  = -\frac{1}{2}\int_{\Omega}u^{2a}\D(\eta^2)dx .
\end{align}
Take $\varphi =u^{a}\eta$ in inequality \eqref{stability interp}. Using \eqref{3.1}-\eqref{3.2}, we obtain
\begin{align*}
 \sqrt{pq}\int_{\Omega}u^{\frac{q-1}{2}}v^{\frac{p-1}{2}}u^{2a}\eta^2dx \leq \; \int_{\Omega}|\nabla\varphi|^2dx\leq  \;\frac{a^2}{2a-1}\int_{\Omega}u^{2a-1}v^p\eta^2dx +C\int_{\Omega}u^{2a}\Big[|\nabla\eta|^2+\vert \Delta(\eta^2) \vert\Big]dx.
\end{align*}
So,
\begin{align}\label{3.3}
\int_{\Omega}u^{\frac{q-1}{2}}v^{\frac{p-1}{2}}u^{2a}\eta^2dx \leq
A^{-1}\int_{\Omega}u^{2a-1}v^p\eta^2dx
 +C\int_{\Omega}u^{2a}\Big[|\nabla\eta|^2+\vert \Delta(\eta^2) \vert\Big]dx.
\end{align}
%Choose now
%$\phi(x)= \varphi_0(R^{-1}x)$ where $\varphi_0 \in C_c^\infty(\Omega)$ such that $\varphi_0
%\equiv 1$ in $B_1$, there holds then
% \begin{equation}\label{3.3}
%\int_{\mathbb{R}^d}u^{\frac{\theta-1}{2}}v^{\frac{p-1}{2}}u^{q+1}\phi^2dx\leq
%\frac{1}{a_1}
%\int_{\mathbb{R}^d}u^qv^p\phi^2dx+\frac{C}{R^2}\int_{B_{2R}}u^{q+1}dx.
%\end{equation}
Similarly, applying inequality \eqref{stability interp} with $\varphi = v^{b}\eta$, $b
\ge 1$, we obtain
\begin{equation}\label{3.4}
\int_{\Omega}u^{\frac{q-1}{2}}v^{\frac{p-1}{2}}v^{2b}\eta^2dx\leq
B^{-1} \int_{\Omega}u^qv^{2b-1}\eta^2dx+C\int_{\Omega}v^{2b}\Big[|\nabla\eta|^2+\vert \Delta(\eta^2) \vert\Big]dx.
\end{equation}
Let $I_1$ denote the left-hand side of \eqref{3.3} (resp. $I_2$ that of \eqref{3.4}). Then,
 \begin{align}
 \label{3.5}
  \begin{split} {A}^\frac{4a}{q+1} I_1+ I_2
 =  \; {A}^\frac{4a}{q+1}\int_{\Omega}u^{\frac{q-1}{2}}v^{\frac{p-1}{2}}u^{2a}\eta^2dx+
  \int_{\Omega}u^{\frac{q-1}{2}}v^{\frac{p-1}{2}}v^{2b}\eta^2dx\\
  \leq \; A^{\frac{4a}{q+1}-1}
\int_{\Omega}u^{2a-1}v^p\eta^2dx+B^{-1}\int_{\Omega}u^qv^{2b-1}\eta^2dx\\
+C\int_{\Omega}(u^{2a}+v^{2b})\Big[|\nabla\eta|^2+\vert \Delta(\eta^2) \vert\Big]dx.
  \end{split}
 \end{align}
Fix now
\begin{equation}\label{3.6}
 b=\frac{p+1}{q+1}a.
\end{equation}
Assume that $a\geq \frac{q+1}{2}$ so that $m=\frac{q+1}{4a}=\frac{p+1}{4b}\in (0,1).$ By Young's inequality, there holds
\begin{align*}
 B^{-1}\int_{\Omega}u^qv^{2b-1}\eta^2dx
= & \; B^{-1}\int_{\Omega}u^{\frac{q-1}{2}}v^{\frac{p-1}{2}}u^{2am}v^{2b(1-m)}\eta^2 dx \\
%\end{align*}\begin{align*}
\leq & \; mB^{-\frac{1}{m}}I_1+(1-m) I_2
\end{align*}
and similarly we have
 \begin{equation}\nonumber
  A^{\frac{4a}{q+1}-1}\int_{\Omega}u^{2a-1}v^p\eta^2dx \leq (1-m)A^{\frac{1}{m}} I_1+mI_2.
   \end{equation}
   Combining the above two estimates with \eqref{3.5}, we derive that
\begin{equation*}
{A}^{\frac{1}{m}}I_1\leq
\left(mB^{-\frac{1}{m}}+(1-m)A^{\frac{1}{m}}\right)I_1
+C\int_{\Omega}(u^{2a}+v^{2b})\Big[|\nabla\eta|^2+\vert \Delta(\eta^2) \vert\Big]dx,
\end{equation*}
hence
\begin{equation*}
m\left((AB)^{\frac{1}{m}}-1\right)
I_1\leq C\int_{\Omega}(u^{2a}+v^{2b})\Big[|\nabla\eta|^2+\vert \Delta(\eta^2) \vert\Big]dx.
\end{equation*}
Thus, if $AB > 1$, there holds
\begin{equation*}
    \int_{\Omega}u^{\frac{q-1}{2}}v^{\frac{p-1}{2}}u^{2a}\eta^2dx  \leq C\int_{\Omega}(u^{2a}+v^{2b})\Big[|\nabla\eta|^2+\vert \Delta(\eta^2) \vert\Big]dx.
\end{equation*}
Using \eqref{3.6} and \eqref{estS}, there holds $u^{2a} \geq Cv^{2b}$ and
$u^{\frac{q-1}{2}}v^{\frac{p-1}{2}}u^{2a}\geq C u^{2a-1} v^p$.
So if $AB> 1$ and $a > \frac{q+1}{4}$,
\begin{equation*}
\int_{\Omega}v^{p}u^{2a-1}\eta^2dx  \leq C\int_{\Omega}u^{2a}\Big[|\nabla\eta|^2+\vert \Delta(\eta^2) \vert\Big] dx,
\end{equation*}
the proof is completed. \end{proof}
Using the above lemma, we obtain the following estimate.
\begin{lem}\label{lmain1}
Assume that $(u,v)\in C^2(B_2)$ is a positive solution of \eqref{1.1} which is stable in $B_2$. Let $\theta\in\left[1,\frac{d}{d-2}\right)$.Then for any $a\ge\frac{q+1}{2}$ satisfying $AB>1,$ there exists $C >0$ depending on $p,q,d,\theta, a$ only such that
\begin{align}\label{3.10}
 \int_{ B_1}u^{2a\theta}dx\leq C.
 \end{align}
\end{lem}
\begin{proof}[Proof of Lemma \ref{lmain1}]
 Let $(u,v)$ be a solution of \eqref{1.1} which is stable in $B_2$ and take $\theta\in\left[1,\frac{d}{d-2}\right)$. Take $\eta \in C_c^\infty(B_2)$. By $L^1$ elliptic regularity theory, there exists $C>0$ such that
\begin{align*}
  \left(\int_{B_2}u^{2a\theta}\eta^{2\theta}\right)^{\frac{1}{\theta}}\leq C\int_{B_2}|\Delta (u^{2a}\eta^2)|.
\end{align*}
Expanding $\Delta (u^{2a}\eta^2)$ in the right-hand side of the above inequality, we find
\begin{align*}
 |\Delta (u^{2a}\eta^2)|\leq (2a)  u^{2a-1}v^p \eta^2 +2a(2a-1)u^{2a-2}|\nabla u|^2\eta^2+ u^{2a} |\Delta(\eta^2)|+4au^{2a-1}|\nabla u||\nabla \eta^2|.
\end{align*}
Then using Young's inequality, there holds
\begin{align}\label{3.7}
  \left(\int_{B_2}u^{2a\theta }\eta^{2\theta}\right)^{\frac{1}{\theta}}\leq C\left[\int_{B_2} u^{2a-1}v^p \eta^2 +\int_{B_2}u^{2a-2}|\nabla u|^2\eta^2+\int_{B_2} u^{2a} \left(|\Delta(\eta^2)|+|\nabla \eta|^2\right)\right].
\end{align}
Multiplying $-\Delta u= v^p$ by $u^{2a-1}\eta^2,$ integrating by parts and applying Young's inequality, we get
\begin{align}\label{3.8}
  \int_{B_2} u^{2a-2}|\nabla u|^2 \eta^2dx\leq  \frac{1}{2a-1}\int_{B_2}  u^{2a-1}v^p \eta^2dx + C\int_{B_2} u^{2a} |\Delta(\eta^2)|dx.
\end{align}
Combining the last two estimates, we obtain
\begin{align*}
  \left(\int_{B_2}u^{2a\theta}\eta^{2\theta}\right)^{\frac{1}{\theta}}\leq C\left[\int_{B_2} u^{2a-1}v^p \eta^2 +\int_{B_2} u^{2a} \left(|\Delta(\eta^2)|+|\nabla \eta|^2\right)\right].
\end{align*}
By Lemma \ref{Newl} with $\Omega=B_2$, we conclude that
\begin{align*}
 \left(\int_{B_2}u^{2a\theta}\eta^{2\theta}\right)^{\frac{1}{\theta}}\leq C\int_{B_2} u^{2a} \left(|\Delta(\eta^2)|+|\nabla \eta|^2\right).
\end{align*}
Take $\varphi$ a cut-off function in $C_c^\infty(B_2)$ verifying $0 \leq \varphi \leq 1$, $\varphi=1$ in $B_1$. Letting $\eta = \varphi^m$, $m\ge 1$, we arrive at
 \begin{equation}\label{3.9}
 \left(\int_{B_2}u^{2a\theta }\varphi^{2\theta m}dx\right)^{\frac{1}{\theta}} \leq C\int_{B_2} u^{2a}\varphi^{2m-2}dx.
 \end{equation}
Let
 \begin{align*}
J_1 := \int_{B_2}u^{2a\theta}\varphi^{2\theta m}dx , \quad  J_2 := \int_{B_2} u^{2a}\varphi^{2m-2}dx.
 \end{align*}
Since $a\ge\frac{q+1}{2},$ we have $q < 2a < 2a\theta$. A direct calculation yields
 \begin{align*}
2a= q\lambda + 2a\theta(1 - \lambda) \quad \mbox{with } \lambda = \frac{2a(\theta-1)}{2a\theta-q} \in (0, 1).
 \end{align*}
Take $m$ large such that $m(1-\theta(1-\lambda))=m\frac{q(\theta-1)}{2a\theta-q}$. %m\frac{(2a-1)(\theta-1)}{2a\theta-q} > 1.$
By H\"older's inequality  and \eqref{3.9}, we get
\begin{eqnarray*}
  J_2 &\leq& J_1^{1 - \lambda} \left(\int_{B_2}u^q\varphi^{\frac{2m(1-\theta(1-\lambda))-2}{\lambda}}dx\right)^\lambda \\
& \leq& \left(CJ_2\right)^{\theta(1 - \lambda)} \left(\int_{B_2}u^qdx\right)^\lambda \\
 & \leq& C'J_2^{\theta(1 - \lambda)},
 \end{eqnarray*}
where in the last inequality we have applied Lemma \ref{l.2.1}.  This implies, using again \eqref{3.9},
\begin{align*}
J_1\leq C^\theta (J_2)^\theta \leq  C.
 \end{align*}
Since $\varphi=1$ in $B_1$, \eqref{3.10} follows.
 \end{proof}
Thanks to the energy estimate \eqref{3.10}, Theorem \ref{prop-gamma0} easily follows.
\begin{proof}[Proof of Theorem \ref{prop-gamma0}]
Indeed, applying \eqref{3.10} with $a=\frac{q+1}2$ and some $\theta\in\left(1,\frac{d}{d-2}\right)$, we deduce that the family of rescaled functions $(u_R)_{R\ge1}$ is bounded in $L^{\tilde q+1}_{loc}(\R^d)$ for some $\tilde q>q$. By the comparison inequality \eqref{estS}, we deduce that  so must be $(v_R)_{R\ge1}$ in $L^{\tilde p+1}_{loc}(\R^d)$ for some $\tilde p>p$. Thanks to the equation \eqref{1.1}, it also follows that $(u_R,v_R)_{R\ge1}$ is bounded in $W^{1,1}_{loc}(\R^d)$. This fact and the $L^{\tilde q+1}_{loc}(\R^d)\times L^{\tilde p+1}_{loc}(\R^d)$ bound imply compactness of $(u_R,v_R)_{R\ge1}$ in $L^{q+1}_{loc}(\R^d)\times L^{p+1}_{loc}(\R^d)$ as claimed.
\end{proof}
The energy estimate \eqref{3.10} also implies the following Liouville type result for positive stable solutions of \eqref{1.1}.

\begin{thm}
\label{main1} Suppose $p\geq q > 1.$
Then, \eqref{1.1} has no stable classical solution if $d< 2+2x_0,$ where $x_0$ is the largest root of the polynomial $H$ given by \eqref{polyH}.
\end{thm}

\begin{proof} First, for $a>\frac{q+1}{2},$ noting $x=a\alpha ,$ we can easily check that $AB>1\Leftrightarrow H(x)<0.$ Furthermore, by \cite[Lemma 3.1]{H} (applied with $p$ in place of $\theta$, $q$ in place of $p$ and where $H(x)=(\frac{\alpha}{2})^4L(\frac{2}{\alpha}x)$ according to the notations in that paper), we have $H(\frac{q+1}{2}\alpha )<0$ and $x_0$ is the unique root of $H$ in $(\frac{q+1}{2}\alpha, +\infty).$  Suppose that $d<2+2x_0.$ Then, there exists $\frac{q+1}{2}<a<\frac{x_0}{\alpha}$ such that $d<2+2a\alpha.$ Let $R>1$. The function defined for $y\in\R^d$ by
$$(u_R(y),v_R(y)) = (R^{\alpha} u(R y), R^{\beta} v(Ry))$$
is a positive stable solution of \eqref{1.1}. Replacing $u$ in \eqref{3.10} by $u_R$, we deduce that
\begin{align}\label{3.11}
 \int_{ B_R}u^{2a\theta}dx\leq CR^{d-2a\alpha\theta},
 \end{align}
for any  $1<\theta <\frac{d}{d-2}.$
Now, as $d<2+2a\alpha,$ we can take $\theta$ close to $\frac{d}{d-2}$ such that $d<2a \alpha\theta$ and then, letting $R\rightarrow\infty$ in \eqref{3.11}, we deduce that $u=0$.\end{proof}
%\medskip
Turning to solutions merely stable outside a compact set, we first make the following observation.
\begin{rem}\label{rem1}
Let $(u,v)$ be a positive classical solution of \eqref{1.1} which is stable outside a compact set $K$ of $\mathbb{R}^d.$ Let $R_0>0$ so that $K\subset B_{R_0}.$ Suppose that $d<2+2x_0,$ where $x_0$ is the largest root of the polynomial $H$ given by \eqref{polyH}. Let $\gamma <\frac{2x_0}{\alpha}\frac{d}{d-2}.$
\begin{enumerate}
\item Using \eqref{3.11}, we have
\begin{equation}\label{3.12}
 \int_{ B_R}u^{\gamma}dx\leq CR^{d-\gamma\alpha}, \quad \mbox{for any}\; R>0 \;\mbox{and}\; B_{2R}\subset \mathbb{R}^d\setminus K.% B_{R_0}.
\end{equation}
\item In addition, all the computations in the proof of Lemma \ref{lmain1} hold true by considering a cut-off function with support in $\Omega=B_{2R}\setminus\overline{B_{R_0}},$ for $R<R_0+3$ (see for example \cite{Far}), to find
\begin{equation}\label{3.int}
 \int_{ R_0+2<|x|<R}u^{\gamma}dx\leq C_1+C_2R^{d-\gamma\alpha}, \quad \mbox{for any}\; R>R_0+3.
\end{equation}
\end{enumerate}
\end{rem}
%\medskip
Following a strategy used in \cite{Far}, we deduce the following decay estimate.
\begin{lem}
   Let $p\ge q\ge1$ be such that $pq>1$ and $d< 2+2x_0,$ where $x_0$ is the largest root of the polynomial $H$ given by \eqref{polyH}. Suppose that $(u,v)\in C^2(\R^d)$ is a nonnegative solution of \eqref{1.1} which is stable outside a compact set and \eqref{not critical} holds. Then,
$$
u(x)=o\left(|x|^{-\alpha}\right), \quad v(x)=o\left(|x|^{-\beta}\right) \quad \text { as }|x| \rightarrow \infty.
$$
\end{lem}

\begin{proof} Fix $\varepsilon>0$. As $d<2+2x_0$ and \eqref{not critical} holds, we have $2<\frac{d}{\alpha}<\frac{2x_0}{\alpha}\frac{d}{d-2}.$ So, Applying 2) of Remark \ref{rem1} with $\gamma=\frac{d}{\alpha}$, there exists $R_{1}>R_{0}$ such that
\begin{equation}\label{3.13}
\left(\int_{|x| \geq R_{1}} u^{\frac{d}{\alpha}}\right)^{\frac{\alpha}{d}}<\varepsilon.
\end{equation}
Consider $y \in \mathbb{R}^{d}$ such that $|y|>2 R_{1}$ and set $\rho=\frac{|y|}{4}$. With this choice, we have
$$
B_{2 \rho}(y) \subset\left\{x \in \mathbb{R}^{d} ;|x|>R_{1}\right\} \subset\left\{x \in \mathbb{R}^{d} ;|x|>R_{0}\right\}.
$$
Remark that $u$ is a solution of the linear equation
$$
\Delta u+\ell(x) u=0 \text { in } {B_{2 \rho}(y)},
$$
where $\ell(x)=\frac{v^{p}}{u}$. By the comparison inequality \eqref{estS}, we have $\ell \leq C u^{\frac{pq-1}{p+1}}=Cu^{\frac2\alpha}.$ In addition, since $\frac{d}{\alpha}<\frac{2x_0}{\alpha}\frac{d}{d-2},$ there exists $\delta$ sufficiently small such that $\frac{2d}{(2-\delta)\alpha}<\frac{2x_0}{\alpha}\frac{d}{d-2} .$ So, by Remark \ref{rem1}, we derive that $\ell \in L^{\frac{d}{2-\delta}}\left(B_{2 \rho}(y)\right)$. Therefore, according to \cite[Theorem 1]{sz}, we obtain
\begin{equation}\label{3.14}
\|u\|_{L^{\infty} (B_{\rho}(y))} \leq C_{S} \rho^{-\frac{d}{2}}\|u\|_{L^{2} (B_{2 \rho}(y))},
\end{equation}
where $C_{S}$ is a constant depending only on $p, d, R_0$ $\rho^{\delta}\|\ell\|_{L^{\frac{d}{2-\delta}}(B_{2 \rho}(y))}$. Now, we have
$$
\rho^{\delta}\|\ell\|_{L^{\frac{d}{2-\delta}}(B_{2 \rho}(y))} \leq C \rho^{\delta}\left(\int_{B_{2 \rho}(y)} u^{\frac{2d}{(2-\delta)\alpha}} d x\right)^{\frac{2-\delta}{d}}.
$$
Applying Remark \ref{rem1} with $\gamma=\frac{2d}{(2-\delta)\alpha},$ it follows that
\begin{equation}
\rho^{\delta}\|\ell\|_{L^{\frac{d}{2-\delta}}(B_{2 \rho}(y))}  \leq C \rho^{\delta} \rho^{(d-\gamma \alpha)\frac{2-\delta}{d}}=C.
\end{equation}
Hence, the constant $C_{S}$ depends only on $p, d$ and $R_{0}$. Using \eqref{3.14} and applying H\"older's inequality, we obtain
$$
\begin{aligned}
\|u\|_{L^{\infty} (B_{\rho}(y))} & \leq C_{S} \rho^{-\frac{d}{2}}\|u\|_{L^{2} (B_{2 \rho}(y))} \\
& \leq C_{S} C \rho^{-\frac{d}{2}} \rho^{d\left(\frac{1}{2}-\frac{\alpha}{d}\right)}\|u\|_{L^{\frac{d}{\alpha}}(B_{2 \rho}(y))} \\
& =C_{1} \rho^{-\alpha}\|u\|_{L^{\frac{d}{\alpha}}(B_{2 \rho}(y))}.
\end{aligned}
$$
Now, using \eqref{3.13}, we get
$$
\begin{aligned}
|u(y)| & \leq\|u\|_{L^{\infty} (B_{\rho}(y))} \\
& \leq C_{1} \rho^{-\alpha} \varepsilon \\
& =C_{2}|y|^{-\alpha} \varepsilon.
\end{aligned}
$$
In summary, we have proven that for any $\varepsilon>0$ there exists $M=2 R_{1}$ such that
$$
y \in \mathbb{R}^{d},|y|>M \Rightarrow|y|^{\alpha}|u(y)| \leq C_{2} \varepsilon.
$$
Therefore, $u(x)=o\left(|x|^{-\alpha}\right)$ as $|x| \rightarrow \infty$ and by the comparison inequality \eqref{estS}, we deduce that $v(x)=o\left(|x|^{-\beta}\right)$, which completes the proof.
\end{proof}

\section{Fast decay solutions}\label{secfast}
In order to go further, we proceed by proving that the solution has in fact faster decay rate. As we shall see, this will be enough to conclude (in the range of parameters studied in the previous section), thanks to Pohozaev's identity.
\begin{thm}\label{prop 2.1}
  Let $d\geq 3,$ $p\geq q >1$ supercritical i.e.  $\frac{1}{p+1}+\frac{1}{q+1}< 1-\frac{2}{d}.$ Then, the system \eqref{1.1} has no classical positive solution $(u,v)$ satisfying
  \begin{equation}\label{Dec u v}
    u(x)=\smallO{\vert x\vert^{-\alpha}},\;v(x)=\smallO{\vert x\vert^{-\beta}},\quad \mbox{as}\; \vert x\vert \rightarrow \infty.
  \end{equation}
\end{thm}
Theorem \ref{prop 2.1} is known, see e.g. Theorem 1.1 in Cheng and Huang \cite{ch}. We provide a short proof for the convenience of the reader, by first establishing the following decay estimates.
\begin{lem}\label{l3}
  Let $d\geq 3,$ $p\geq q>1$ and $(u,v)$ a classical positive solution to the system \eqref{1.1} such that
  \begin{equation}\label{Dec u}
    u(x)=\smallO{\vert x\vert^{-\alpha}},\quad \mbox{as}\; \vert x\vert \rightarrow \infty.
  \end{equation}
    Then, for any small $\epsilon >0$, there exists two positive constants $c$ and $C$ such that
   \begin{equation}\label{F Dec}
     c \vert x\vert ^{2-d}\leq u(x)\leq C\vert x\vert ^{2-d+\epsilon} \quad\mbox{and} \quad\vert\nabla u(x)\vert\leq C\vert x\vert ^{1-d+\epsilon}\quad \mbox{for}\; \vert x\vert\ge 1.
   \end{equation}
   In addition,
   \begin{equation}\label{F Dec2}
     \quad \vert\nabla v(x)\vert=o(\vert x\vert ^{-\beta-1}) \quad \mbox{as}\; \vert x\vert \rightarrow \infty.
   \end{equation}
\end{lem}
\begin{proof}
 Since $u$ is a non-trivial, positive and superharmonic function, the first inequality in \eqref{F Dec} is a standard comparison result. For the proof see e.g. \cite[Lemma 2.1]{sz}.

Now, by Lemma \ref{Soup}, we have $$-\Delta u= v^p\leq \left(\frac{p+1}{q+1}\right)^{\frac{p}{p+1}}u^{\frac{2}{\alpha}}u.$$ Let $0<\epsilon<\frac12.$ By the decay estimate \eqref{Dec u}, there exists $R_\epsilon >1$ such that $\left(\frac{p+1}{q+1}\right)^{\frac{p}{p+1}}u^{\frac{2}{\alpha}}\le \frac{\epsilon^2}{\vert x\vert^2},$ for any $\vert x\vert >R_\epsilon.$ Therefore,
  $$-\Delta u-\frac{\epsilon^2}{\vert x\vert^2}u\leq 0,\; \mbox{for any}\; \vert x\vert >R_\epsilon.$$
  Now, we can easily check that the functions $f_{a}:=\vert x\vert^{-a}$,  $f_{b}:=\vert x\vert^{-b}$ are solutions of
  \begin{equation}\label{noy}
    -\Delta f -\frac{\epsilon^2}{\vert x\vert^2}f=0 \quad \mbox{in}\; \mathbb{R}^d\backslash \{0\},
  \end{equation}
  where $a= \frac{d-2}{2}+ \sqrt{(\frac{d-2}{2})^2-\epsilon^2}$ and $b=\frac{d-2}{2}- \sqrt{(\frac{d-2}{2})^2-\epsilon^2}$. Let $R> R_\epsilon$ and consider $w(x)=R_\epsilon^{a}\|u\|_\infty f_{a}(x) +\epsilon^2 R^{b-\alpha}f_b(x).$
  Then, we have
  \begin{align}\label{}
 \left\{\begin{array}{lll}\left(-\Delta  -\frac{\epsilon^2}{\vert x\vert^2}\right)(u-w)\leq 0,\quad \mbox{in}\;A_{R_\epsilon}^R:=B_R\backslash \overline{B_{R_\epsilon}}, \\
u-w\leq 0,\quad \mbox{on}\; \partial A_{R_\epsilon}^R.
\end{array}
\right.
\end{align}
By the maximum principle, we deduce that $u\leq w$ in $A_{R_\epsilon}^R$ for any $R>A_{R_\epsilon}^R.$ Choosing $\epsilon>0$ so small that $b-\alpha<0$ and letting $R\rightarrow \infty,$ we conclude that for $\vert x\vert\ge R_\epsilon$,
$$u(x)\leq C\vert x\vert ^{-a}\leq C\vert x\vert ^{2-d+\epsilon}.$$
To obtain the gradient estimates, we scale and apply standard elliptic regularity. Precisely, given $z\in\R^d\setminus\{0\}$ and $\rho=\vert z\vert/2$, let $(\tilde u(x), \tilde v(x))=(\rho^{\alpha}u(z+\rho x), \rho^{\beta}v(z+\rho x))$ for $x\in\R^d$. Then, $(\tilde u,\tilde v)$ still solves the Lane-Emden system and by standard elliptic regularity
$$
\vert\nabla \tilde u(0)\vert\le C\left(\Vert\tilde v^p\Vert_{L^\infty(B_1)}+\Vert\tilde u\Vert_{L^\infty(B_1)}\right)\le C'\Vert\tilde u\Vert_{L^\infty(B_1)},
$$
where the last inequality follows from the comparison inequality Lemma \ref{Soup} and the fact that $pq\ge 1$. Hence, for $\vert z\vert\ge1$,
$$
\vert\nabla  u(z)\vert\le C\vert z\vert^{1-d+\epsilon},
$$
as claimed. The gradient estimate on $v$ follows similarly.
\end{proof}

The last crucial ingredient in the proof of Theorem \ref{prop 2.1} is the following identity of Pohozaev-type (see \cite{Mi, pz, souplet}).
\begin{lem}
  Let $p, \,q>0$ and $(u, v)$ a positive classical solution of \eqref{1.1}. Then, for any $R>0$, there holds
  \begin{align}\label{Poh}
\begin{split}
&\; \left(\frac{d}{p+1}-a_1\right)\int_{B_R} v^{p+1}  +\left(\frac{d}{q+1}-a_2\right) \int_{B_R} u^{q+1}  \\
=& R^d\int_{S^{d-1}}\left(\frac{u^{q+1}}{q+1}+\frac{v^{p+1}}{p+1}\right)(R\theta)d\sigma(\theta) +R^{d-1}\int_{S^{d-1}}\left(a_1u\partial_rv+a_2v\partial_ru\right)(R\theta)d\sigma(\theta)\\
&+R^d\int_{S^{d-1}}\left(\partial_ru\partial_rv
-\frac{\nabla_\theta u\cdot\nabla_\theta v}{R^2}\right)(R\theta)d\sigma(\theta),
\end{split}
 \end{align}
 where $a_2,a_2 \in \mathbb{R}$ satisfy $a_1+a_2=d-2$.
\end{lem}

\begin{proof}[Proof of Theorem \ref{prop 2.1}.] We argue by contradiction and suppose that there exists a classical positive solution $(u,v)$ to the system \eqref{1.1} which decays faster than the homogeneous solution.
Thanks to \eqref{F Dec}, since $q+1>\frac{d}{d-2}$, $u\in L^{q+1}(\R^d)$. By the comparison inequality \eqref{estS}, $v\in L^{p+1}(\mathbb{R}^d)$. We deduce that the left-hand side of \eqref{Poh} converges as $R\to+\infty$. Thanks to \eqref{F Dec} and \eqref{F Dec2}, we also deduce that all integrals on the right-hand side of \eqref{Poh} converge to $0$ as $R\to+\infty$. Hence,
 \begin{equation*}
   \left(\frac{d}{p+1}-a_1\right)\int_{\R^d} v^{p+1} dx +\left(\frac{d}{q+1}-a_2\right) \int_{\R^d} u^{q+1} dx=0,
 \end{equation*}
 for any $a_1,\,a_2$ satisfying $a_1+a_2=d-2.$ Take $a_1= \frac{d}{p+1},$ it follows that
 \begin{equation*}
   \left(\frac{d}{q+1}-a_2\right) \int_{\R^d} u^{q+1} dx=0.
 \end{equation*}
 Since $\frac{1}{p+1}+\frac{1}{q+1}<\frac{d-2}{2},$ we have
 $$\left(\frac{d}{q+1}-a_2\right)=\frac{d}{q+1}+\frac{d}{q+1}-(d-2)<0.$$
 Therefore, $u\equiv 0$ in $\mathbb{R}^d,$ a contradiction. %So, we are done in the supercritical case.
\end{proof}

From the above two sections, we conclude the proof of Theorem \ref{tbth4}. Moreover, as shown in page 277 of \cite{HHM}, we have $x_0>4,$ $\forall\, p\geq q>1.$ Hence, Theorem \ref{tbth3} is a consequence of Theorem \ref{tbth4} and Theorem \ref{main1}.

%Also, note that if $(u,v)$ solves the Lane-Emden system in $\Omega$ and $m=1+1/p$, then\footnote{using the convention that  $\vert t\vert^{m-2}t=0$ if $t=0$ and $m>1$} $v=\vert\Delta u\vert^{m-2}\Delta u$.
\section{Stable homogeneous solutions}
This section is not needed for the proofs of the theorems stated in the introduction. It is written in order to classify homogeneous solutions and to prepare the reader for the more delicate iteration method of the next section.
Let us restrict to the class of stable homogeneous solutions to the Lane-Emden system \eqref{1.1}, in the supercritical case \eqref{supercritical}. Let $u(r\theta)=r^{-\alpha}f(\theta),\;v(r\theta)=r^{-\beta}g(\theta),_; r> 0,\,\theta \in S^{d-1}$ be a stable homogeneous solution of \eqref{1.1}. Then $(f,g)$ satisfies
\begin{equation}\label{1.8}
-\Delta_\theta f  = -\lambda f+  \vert g\vert^{p-1}g ,\quad
-\Delta_\theta g  = -\mu g+ \vert f\vert^{q-1}f \;\quad\text{in } \;S^{d-1}.
\end{equation}
Moreover, from Lemma \ref{lemma4}, the stability of the homogeneous solution $(u,v)$, it follows that
\begin{equation}\label{1.9}
\int_{S^{d-1}} \abs {\nabla_\theta \phi}^2 \,d\gs +H\int_{S^{d-1}} \phi^2 \,d\gs\geq \sqrt{pq}\int_{S^{d-1}} \vert f\vert^{\frac{q-1}{2}}\vert g\vert^{\frac{p-1}{2}}\phi^2\,d\gs,\quad \forall\, \phi\in C^{\infty}(S^{d-1}).
\end{equation}
We prove the following theorem.
\begin{thm}\label{thm-homo} Let $p\geq q \ge 1$ such that $pq>1$. There exists a stable homogeneous solution to \eqref{1.1}  if and only if \eqref{1.7} does not hold.
\end{thm}
\begin{proof} [Proof of Theorem \ref{thm-homo}]
By the discussion in the introduction, it suffices to prove the "only if" part of Theorem \ref{thm-homo}. For this, we proceed by contradiction. Suppose that $H^2<pq \lambda \mu$ and $(u,\,v)=(r^{-\alpha}f(\theta),\,r^{-\beta}g(\theta))$ is a stable homogeneous solution of \eqref{1.1}. As in Section \ref{sec3}, we denote $A= \sqrt{pq}\frac{2a-1}{a^2}$ and $B=\sqrt{pq}\frac{2b-1}{b^2},$ with $b=\frac{p+1}{q+1}a.$ Multiplying the equations of the system \eqref{1.8} by $\sqrt{pq}\vert f\vert^{2a-2}f$ and $\sqrt{pq}\vert g\vert^{2b-2}g$ respectively and integrating over $S^{d-1},$ there holds
\begin{equation}\label{4.1}
A \int_{S^{d-1}} \abs {\nabla_\theta (\vert f\vert ^{a-1}f)}^2 \,d\gs +\sqrt{pq}\lambda \int_{S^{d-1}} \vert f\vert^{2a} \,d\gs=\sqrt{pq} \int_{S^{d-1}}  \vert g\vert^{p-1}g\vert f\vert^{2a-2}f\,d\gs,
\end{equation}
and
\begin{equation}\label{4.2}
B \int_{S^{d-1}} \abs {\nabla_\theta (\vert g\vert ^{b-1}g)}^2 \,d\gs +\sqrt{pq}\mu \int_{S^{d-1}} \vert g\vert^{2b} \,d\gs=\sqrt{pq} \int_{S^{d-1}} \vert f\vert^{q-1}f\vert g\vert^{2b-2}g\,d\gs.
\end{equation}
Take $a>\frac{q+1}{4},$ $b=\frac{p+1}{q+1}a>\frac{p+1}{4}$ and let $m:=\frac{q+1}{4a}=\frac{p+1}{4b}\in (0,\,1).$ By H\"older's inequality, we get
\begin{equation}\label{4.3}
\sqrt{pq} \int_{S^{d-1}}  \vert g\vert^{p-1}g\vert f\vert^{2a-2}f\,d\gs\leq  \sqrt{pq} \int_{S^{d-1}} \vert f\vert^{\frac{q-1}{2}}\vert g\vert^{\frac{p-1}{2}}\vert f\vert^{2a(1-m)}\vert g\vert^{\frac{p+1}{2}}\leq \sqrt{pq}I_1^{1-m}I_2^m
\end{equation}
and
\begin{equation}\label{4.4}
\sqrt{pq} \int_{S^{d-1}} \vert f\vert^{q-1}f\vert g\vert^{2b-2}g\,d\gs\leq  \sqrt{pq} \int_{S^{d-1}} \vert f\vert^{\frac{q-1}{2}}\vert g\vert^{\frac{p-1}{2}}g^{2b(1-m)}\vert f\vert ^{\frac{q+1}{2}}\leq \sqrt{pq}I_1^{m}I_2^{1-m},
\end{equation}
where $I_1:=\displaystyle{\int_{S^{d-1}}} \vert f\vert^{\frac{q-1}{2}+2a}\vert g\vert^{\frac{p-1}{2}}\,d\gs$ and $I_2:=\displaystyle{\int_{S^{d-1}}} \vert f\vert^{\frac{q-1}{2}}\vert g\vert^{\frac{p-1}{2}+2b}\,d\gs.$
 Combining \eqref{4.1}--\eqref{4.4}, we obtain
\begin{equation*}
\begin{split}
&\left(A \int_{S^{d-1}} \abs {\nabla_\theta (\vert f\vert ^{a-1}f)}^2 \,d\gs +\sqrt{pq}\lambda \int_{S^{d-1}} \vert f\vert^{2a} \,d\gs\right)\left(B \int_{S^{d-1}} \abs {\nabla_\theta (\vert g\vert ^{b-1}g)}^2 \,d\gs +\sqrt{pq}\mu \int_{S^{d-1}} \vert g\vert^{2b} \,d\gs\right)\\
&\leq pq I_1I_2.
\end{split}
\end{equation*}
On the other hand, testing $\vert f\vert^{a-1}f$ and $\vert g\vert^{b-1}g$ in the stability inequality \eqref{1.9}, there holds
\begin{equation*}
\sqrt{pq}I_1\leq   \int_{S^{d-1}} \abs {\nabla_\theta (\vert f\vert^{a-1}f)}^2 \,d\gs +H\int_{S^{d-1}} \vert f\vert^{2a} \,d\gs,\quad \sqrt{pq}I_2\leq   \int_{S^{d-1}} \abs {\nabla_\theta (\vert g\vert^{b-1}g)}^2 \,d\gs +H\int_{S^{d-1}} \vert g\vert^{2b} \,d\gs.
\end{equation*}
From the last two inequalities, it follows that
\begin{equation*}
\begin{split}
&\left(A \int_{S^{d-1}} \abs {\nabla_\theta (\vert f\vert^{a-1}f)}^2 \,d\gs +\sqrt{pq}\lambda \int_{S^{d-1}} \vert f\vert^{2a} \,d\gs\right)\left(B \int_{S^{d-1}} \abs {\nabla_\theta (\vert g\vert^{b-1}g)}^2 \,d\gs +\sqrt{pq}\mu \int_{S^{d-1}} \vert g\vert^{2b} \,d\gs\right)\\
&\leq \left(\int_{S^{d-1}} \abs {\nabla_\theta (\vert f\vert^{a-1}f)}^2 \,d\gs +H\int_{S^{d-1}} \vert f\vert^{2a} \,d\gs\right)\left(\int_{S^{d-1}} \abs {\nabla_\theta (\vert g\vert^{b-1}g)}^2 \,d\gs +H\int_{S^{d-1}} \vert g\vert^{2b} \,d\gs\right),
\end{split}
\end{equation*}
or equivalently,
\begin{equation}\label{4.5}
\begin{split}
( AB-1) \int_{S^{d-1}} \abs {\nabla_\theta (\vert f\vert^{a-1}f)}^2 \,d\gs\int_{S^{d-1}} \abs {\nabla_\theta (\vert g\vert^{b-1}g)}^2 \,d\gs\\ +(\sqrt{pq}A\mu-H)\int_{S^{d-1}} \abs {\nabla_\theta (\vert f\vert^{a-1}f)}^2 \,d\gs \int_{S^{d-1}} \vert g\vert^{2b} \,d\gs&\\
+(\sqrt{pq}B\lambda-H)\int_{S^{d-1}} \abs {\nabla_\theta (\vert g\vert^{b-1}g)}^2 \,d\gs \int_{S^{d-1}} \vert f\vert^{2a} \,d\gs &\\
+(pq\lambda \mu-H^2)\int_{S^{d-1}} \vert f\vert^{2a} \,d\gs\int_{S^{d-1}} \vert g\vert^{2b} \,d\gs\leq 0.
\end{split}
\end{equation}
 Choose $a=\frac{q+1}2$ and so $b=\frac{p+1}{2}.$ Since $p\geq q >1,$ we have
  $$AB-1= \frac{1}{(q+1)^2(p+1)^2}\left(16p^2q^2-(q+1)^2(p+1)^2\right)>0. $$
  Hence, applying Theorem \ref{main1}, it follows that $f=g=0$ if $d-4\le \alpha+\beta.$ So, we can assume that $d-4> \alpha+\beta.$
 In addition, since $p\geq q>1,$ we have $\frac{\beta}{\alpha}A^2=\frac{16pq^3}{(p+1)(q+1)^3}=\frac{2p}{p+1}\left(\frac{2q}{q+1}\right)^3\geq 1.$ Recall that $\alpha\geq \beta,$ so that $\frac{\lambda}{\mu}=\frac{\alpha}{\beta}\frac{d-2-\alpha}{d-2-\beta}\leq \frac{\alpha}{\beta}\leq A^2. $ Since $H^2<pq \lambda \mu,$ we find $H<\sqrt{pq}\sqrt{\lambda}\sqrt{\mu}\leq\sqrt{pq}A \mu.$ Using $d-4> \alpha+\beta,$ we can easily check that $A\mu \leq B\lambda$ and so $H< \sqrt{pq}A\mu\leq\sqrt{pq}B\lambda.$ We have just proved that all the constants appearing in \eqref{4.5} are positive, hence $f=g=0$. So, we are done.\end{proof}

%\section*{Acknowledgments}

\section{The case $d-4>\alpha+\beta$: an iteration method in the spirit of De Giorgi}

\begin{lem}\label{l1}
Let $(u,v)$ be a  positive solution of \eqref{1.1} satisfying \eqref{stability interp improved} in $B_2$. Assume that $d-4>\alpha+\beta$ and \eqref{1.7} holds.
Then, there exist constants $C>0$ and $\sigma\in (0,1)$ depending on $d,p,q$ only such that letting $U=u/u_s$, $V=v/v_s$, $a=(q+1)/2$ and $b=(p+1)/2$, there holds
\begin{equation}\label{dgestimate}
\Vert \vert U\vert^{a-1}U\Vert_{C^\sigma(B_{1})}+ \Vert \vert V\vert^{b-1}V\Vert_{C^\sigma(B_{1})}\le C. %\int_{B_2} |x|^{-d}\Big(|U|^{q+1}+|V|^{p+1}\Big).
\end{equation}
\end{lem}

\begin{proof}  We divide our proof into three steps.

\noindent {\bf Step 1: Basic identities.}
Let $a\ge1$, $\varphi\in C^1_c(B_2)$, multiply the first equation of the system by $\vert u\vert^{2a-2}u\varphi^2$ and integrate. The left-hand side is equal to
\begin{eqnarray*}
 \int \nabla u\cdot\nabla (\vert u\vert^{2a-2}u\varphi^2)  &=& (2a-1)\int \vert u\vert ^{2a-2}\vert \nabla u\vert^2 \varphi^2+ \int \vert u\vert^{2a-2}u \nabla u\nabla \varphi^2 \\
   &=& \frac{2a-1}{a^2}\int \vert \nabla (\vert u\vert^{a-1}u)\vert^2\varphi^2+\frac{1}{2a}\int \nabla \vert u \vert^{2a} \nabla \varphi^2 \\
   &=& \frac{2a-1}{a^2}\int \vert \nabla (\vert u\vert^{a-1}u\varphi)\vert^2-\frac{2a-1}{a^2}\int \nabla \varphi \nabla (\vert u \vert^{2a}\varphi)+\frac{1}{2a}\int \nabla \vert u \vert^{2a} \nabla \varphi^2\\
   &=&\frac{2a-1}{a^2}\int \vert\nabla (\vert u\vert^{a-1}u\varphi)\vert^2+\int \vert u\vert^{2a}\left(\frac{2a-1}{a^2}\varphi\Delta\varphi-\frac1{2a}\Delta\varphi^2\right).
\end{eqnarray*}
Next, choose $\varphi=\varphi_0\psi$, where $\varphi_0(x)=\vert x\vert^{\alpha a-\frac{d-2}{2}}$ and $\psi\in C^1_c(\R^d\setminus\{0\})$. Then,
$$
\Delta\varphi=(\Delta\varphi_0)\psi +2\nabla\varphi_0\cdot\nabla\psi+ \varphi_0\Delta\psi,
$$
and
$$
\Delta\varphi^2=(\Delta\varphi_0^2)\psi^2 +2\nabla\varphi_0^2\cdot\nabla\psi^2+ \varphi_0^2\Delta\psi^2.
$$
Let $R$ (for rest) be defined by
$$
\frac{2a-1}{a^2}\varphi\Delta\varphi-\frac1{2a}\Delta\varphi^2 = \left(\frac{2a-1}{a^2}\varphi_0\Delta\varphi_0-\frac1{2a}\Delta\varphi_0^2\right)\psi^2-R,
$$
where
\begin{equation}\label{RR}
R=\frac{1}{2a^2}\nabla \varphi_0^2\cdot \nabla \psi^2+\Big(\frac{1}{a}|\nabla \psi|^2-\frac{a-1}{a^2}\psi\Delta \psi  \Big)\varphi_0^2.
\end{equation}
By direct computation, we have
\begin{eqnarray*}
  \frac{2a-1}{a^2}\varphi_0\Delta\varphi_0-\frac1{2a}\Delta\varphi_0^2 &=& \left[\frac{2a-1}{a^2}\left((\alpha a)^2-\left(\frac{d-2}{2}\right)^2\right)-\frac{1}{2a}\left(2\alpha a(2\alpha a-(d-2))\right)\right]\frac{\varphi_0^2}{\vert x\vert^2}  \\
   &=&\left[(2a-1)\alpha ^2-\frac{2a-1}{a^2}\left(\frac{d-2}{2}\right)^2-2\alpha ^2 a+\alpha (d-2)\right] \frac{\varphi_0^2}{\vert x\vert^2}\\
   &=& \left[\lambda-\frac{2a-1}{a^2}\left(\frac{d-2}{2}\right)^2\right]\frac{\varphi_0^2}{\vert x\vert^2}.
\end{eqnarray*}
So, we just proved that for $\varphi=\varphi_0\psi$,
$$
\int \nabla u\cdot\nabla (\vert u\vert^{2a-2}u\varphi^2) = \frac{2a-1}{a^2}\int\vert\nabla(\vert u\vert^{a-1}u\varphi)\vert^2+\left(\lambda-\frac{2a-1}{a^2}\left(\frac{d-2}{2}\right)^2\right)\int\frac{(\vert u\vert^a\varphi)^2}{\vert x\vert^2}
-\int R \vert u\vert^{2a}.
$$
The left hand-side of the above identity is equal to
$$
\int  \vert v \vert^{p-1}v\vert u\vert^{2a-2}u\varphi^2.
$$
Introducing the quadratic forms
$$
    Q_1(f)=\int\left(\vert\nabla f\vert^2-\left(\frac{d-2}{2}\right)^2\frac{f^2}{\vert x\vert^2}\right), \quad Q_2(f)=\int\frac{f^2}{\vert x\vert^2}
$$
and the numbers $Q_1=Q_1(\vert u\vert^{a-1}u\varphi)$, $Q_2=Q_2(\vert u\vert^{a-1}u\varphi)$, $S=\int  \vert v \vert^{p-1}v\vert u\vert^{2a-2}u\varphi^2$, $T=\int R \vert u\vert^{2a}$, we conclude that
\begin{equation}\label{h1}
\frac{2a-1}{a^2}Q_1 + \lambda Q_2=S+T.
\end{equation}
The choice $a=\frac{q+1}{2}$ makes the expression of $\varphi_0(x)=\vert x\vert^{\frac{(p+1)(q+1)}{pq-1}-\frac{d-2}{2}}$ symmetric in the variables $(p,q)$. Hence, choosing $b=\frac{p+1}{2}$, we may assert that for the {\it same} functions $\varphi=\varphi_0\psi$ and $R$, there holds
\begin{equation}\label{h2}
\frac{2b-1}{b^2}Q_1' + \mu Q_2'=S'+T',
\end{equation}
where $Q_1'=Q(\vert v\vert^{b-1}v\varphi)$, $Q_2'=Q_2(\vert v\vert^{b-1}v\varphi)$, $S'=\int \vert u\vert^{q-1}u\vert v\vert^{2b-2}v\varphi^2$ and $T'=\int R \vert v\vert^{2b}$. For this choice of $a$ and $b$ we have $S=S'$.

\medskip

\noindent {\bf Step 2: $Q_1+Q_1'\leq |T|+|T'|$ and $Q_2+Q_2'\leq |T|+|T'|$.}

\medskip

Multiplying \eqref{h1} and \eqref{h2} we find
$$
\left(\frac{4q}{(q+1)^2}Q_1 + \lambda Q_2\right)\left(\frac{4p}{(p+1)^2}Q_1' + \mu Q_2'\right)\le (S+T)(S+T')= S^2+S(T+T')+TT'.
$$
If $\varepsilon>0$ is small enough, the above inequality yields
\begin{equation}\label{h3}
\left(\frac{4q}{(q+1)^2}Q_1 + \lambda Q_2\right)\left(\frac{4p}{(p+1)^2}Q_1' + \mu Q_2'\right)\le  (1+\varepsilon)S^2+C(\varepsilon)(T^2+T'^2).
\end{equation}
Now, recalling the choice $a=\frac{q+1}{2}, b=\frac{p+1}{2}$, the Cauchy-Schwarz inequality yields
$$
\left\vert\int \vert u\vert^{q-1}u\vert v\vert^{2b-2}v\varphi^2\right\vert\le \int \vert u\vert^{q}\vert v\vert^{p}\varphi^2 \le
\left(\int
\vert u\vert^{\frac{q-1}{2}}\vert v\vert^{\frac{3p+1}2}\varphi^2
\right)^{\frac12}
\left(\int
\vert u\vert^{\frac{3q+1}2}\vert v\vert^{\frac{p-1}{2}}\varphi^2
\right)^{\frac12}.
$$

Applying \eqref{stability interp improved} with test functions $\vert v\vert^{b-1}v\varphi$ and $\vert u\vert^{a-1}v\varphi$, we deduce that
$$
\left\vert\int \vert u\vert^{q-1}u\vert v\vert^{2b-2}v\varphi^2\right\vert\le\frac1{\sqrt{pq}}
\left(\int\vert\nabla(\vert v\vert^{b-1}v\varphi)\vert^2-\frac{\gamma^2}{4}\int\frac{(\vert v\vert^b\varphi)^2}{\vert x\vert^2}\right)^{1/2}\left(\int\vert\nabla(\vert u\vert^{a-1}u\varphi)\vert^2-\frac{\gamma^2}{4}\int\frac{(\vert u\vert^a\varphi)^2}{\vert x\vert^2}\right)^{1/2}.
$$
Similarly,
$$
\left\vert\int  \vert v \vert^{p-1}v\vert u\vert^{2a-2}u\varphi^2\right\vert\le\frac1{\sqrt{pq}}
\left(\int\vert\nabla(\vert v\vert^{b-1}v\varphi)\vert^2-\frac{\gamma^2}{4}\int\frac{(\vert v\vert^b\varphi)^2}{\vert x\vert^2}\right)^{1/2}\left(\int\vert\nabla(\vert u\vert^{a-1}u\varphi)\vert^2-\frac{\gamma^2}{4}\int\frac{(\vert u\vert^a\varphi)^2}{\vert x\vert^2}\right)^{1/2}.
$$
Multiplying both inequalities, we find
\begin{equation}\label{es}
S^2  \le \frac1{pq}(Q_1+HQ_2)(Q_1'+HQ_2').
\end{equation}
We plug this last estimate into \eqref{h3} to find
\begin{multline*}
\left(AB-(1+\varepsilon)\right)Q_1Q_1'+ \left(pq\lambda\mu-(1+\varepsilon)H^2\right)Q_2Q_2'+\left(\sqrt{pq}A\mu-(1+\varepsilon)H \right)Q_1Q_2'\\
+\left(\sqrt{pq}B\lambda-(1+\varepsilon)H \right)Q_1'Q_2 \le C(T^2+T'^2),
\end{multline*}
where, as above, we denote $A:= \sqrt{pq}\frac{2a-1}{a^2}=\sqrt{pq}\frac{4q}{(q+1)^2}$ and $B:=\sqrt{pq}\frac{2b-1}{b^2}=\sqrt{pq}\frac{4p}{(p+1)^2}.$
By the computations made at the end of the previous section, we have $AB>1,$ $\sqrt{pq}A\mu>H$ and $\sqrt{pq}B\lambda>H.$
Then, using \eqref{1.7} and choosing $\varepsilon>0$ small enough we find that all coefficients of the above inequality are positive. This yields
$$
\max\{Q_1Q_1', Q_2Q_2', Q_1Q_2', Q_1'Q_2\} \leq C(T^2+T'^2).
$$
Combing back to \eqref{es} we find $S^2\leq C(T^2+T'^2)$, so $|S|\leq C(|T|+|T'|)$.

\medskip

\noindent{\bf Step 3: Conclusion.}

\medskip

Let us start with the inequality $Q_2+Q_2'\leq C(|T|+|T'|)$ which we established in Step 2 above. Setting
$$
U=\frac{u}{u_s}\quad\mbox{ and }V=\frac{v}{v_s},
$$
we find
$$
\int |x|^{-d}\Big( |U|^{q+1}+|V|^{p+1}\Big) \psi^2\leq C \int R \Big(|U|^{q+1}+|V|^{p+1}\Big).
$$
Take first $\psi(x)=\psi_1(x)$ a standard cut-off function which is identical to $1$ in $B_1$ and vanishes outside of $B_2$. Then, the expression of $R$ in \eqref{RR} and the above estimate yield
$$
\int_{B_1} |x|^{-d}\Big( |U|^{q+1}+|V|^{p+1}\Big) \leq C\int_{B_2\setminus B_1} |x|^{-d}\Big(|U|^{q+1}+|V|^{p+1}\Big),
$$
that is,
$$
\int_{B_1} |x|^{-d}\Big( |U|^{q+1}+|V|^{p+1}\Big) \leq c\int_{B_2} |x|^{-d}\Big(|U|^{q+1}+|V|^{p+1}\Big),
$$
for some $c\in (0,1)$ which depends only on $p, q$ and $d$.
By rescaling and iterating, we deduce that there exist $C, C'>0$ and $\sigma\in (0,1)$ depend only on $p, q$ and $d$ such that
\begin{equation}\label{hy0}
\int_{B_r} |x|^{-d}\Big(  \vert U\vert^{q+1}+\vert V\vert^{p+1}\Big) \leq Cr^{2\sigma}\int_{B_2\setminus B_1} |x|^{-d}\Big(|U|^{q+1}+|V|^{p+1}\Big)\le  C'r^{2\sigma}\quad\mbox{ for all }r\in (0,1),
\end{equation}
 where we used the energy estimate \eqref{3.10} with $(a,\theta)=\left(\frac{q+1}2,1\right)$ and the comparison between components \eqref{estS} in the last inequality.

We now turn to the inequality $Q_1+Q_1'\leq C(|T|+|T'|)$ in which we take $\psi(x)=\psi_1(x/r)$. Using \eqref{hy0} we deduce
\begin{equation}\label{hy}
\int_{B_r} \Big|\nabla \big(  \vert U\vert^{a-1}U|x|^{-\frac{d-2}{2}}  \big)  \Big|^2+\Big|\nabla \big( \vert V\vert^{b-1}V|x|^{-\frac{d-2}{2}}  \big)  \Big|^2\leq \tilde Cr^{2\sigma},
\end{equation}
where, as before $\tilde C$ depends only on $p$, $q$ and $d$.
The identity
$$
\int\Big|\nabla \big(h |x|^{-\frac{d-2}{2}}  \big)  \Big|^2=\Big( \frac{d-2}{2}\Big)^2 \int\frac{h^2}{|x|^2}+\int |x|^{-d+2}\big|\nabla h\big|^2
$$
applied to \eqref{hy} leads us to
\begin{equation}\label{hy1}
r^2 \fint_{B_r}  \Big( \big|\nabla  \vert U\vert^{a-1}U   \big|^2+\big|\nabla  \vert V\vert^{b-1}V    \big|^2\Big) \leq C_1 r^{2\sigma}.
\end{equation}
By the invariance to translation, the above estimate holds for any ball $B_r(x)\subset B_1$, $r\in (0,1)$. If we set $\zeta= \vert U\vert^{a-1}U$ and $\zeta_{x,r}$ the spherical average of $\zeta$ over $B_r(x)$, from Poincar\' e-Wirtinger and \eqref{hy1} we find
$$
\fint_{B_r(x)}|\zeta-\zeta_{x,r}|\leq \left( \fint_{B_r(x)}|\zeta-\zeta_{x,r}|^2 \right)^{1/2}\leq C_1 \left(r^2 \fint_{B_r(x)}|\nabla \zeta |^2 \right)^{1/2}\leq C_2r^\sigma.
$$
The characterization of H\"older functions due to Campanato yields
\begin{equation}\label{HC}
|\zeta(x)-\zeta(y)|\leq C_3|x-y|^\sigma\quad\mbox{ for all }\,  x, y\in B_{1/2}.
\end{equation}
A similar inequality holds for $\xi= \vert V\vert^{b-1}V$ and this concludes our proof.
\end{proof}

\section{The case $d-4>\alpha+\beta$: asymptotics of solutions stable outside a compact set}

\begin{lem}\label{l2}
Let $(u,v)$ be a positive solution of \eqref{1.1} satisfying \eqref{stability interp improved}. Assume that \eqref{1.7} holds.
Then,
\begin{equation}\label{asymp}
|x|^\alpha u(x)\to 0 \quad\mbox{ and }\quad |x|^\beta v(x)\to 0\quad\mbox{ as }|x|\to \infty.
\end{equation}
\end{lem}
\begin{proof}
Without loosing any generality, we may assume that $(u, v)$ satisfies \eqref{stability interp improved} with $K=B_1$.
Let $z\in \R^d\setminus B_2$ and $\rho= |z|/4$. Then the pair $(\tilde u, \tilde v)$ defined as
\begin{equation}\label{hy00}
\tilde u(x)=\rho^\alpha u(z+\rho x), \quad \tilde v(x)=\rho^\beta v(z+\rho x)
\end{equation}
is a solution of \eqref{1.1} satisfying \eqref{stability interp improved} in $B_2$. Letting $\tilde U=\tilde u/u_s$ $\tilde V=\tilde v/v_s$, $a=(q+1)/2$ and $b=(p+1)/2$, we have by Lemma \ref{l1} that $\tilde U^a$ and $\tilde V^b$ are H\"older continuous in $B_{1/2}$.  In particular \eqref{HC} yields
$$
|\tilde U^a(x)-\tilde U^a (y)|\leq C|x-y|^\sigma\quad\mbox{ for all }\,  x, y\in B_{1/2}.
$$
Letting $X=z+\rho x$ and $Y=z+\rho y$, the above estimate implies
$$
\big|U^a(X)-U^a(Y)\big|\leq C\rho^{-a\alpha-\sigma}|X-Y|^\sigma \quad\mbox{ for all }\,  X, Y\in B_{\rho/2}(z),
$$
where, as in Lemma \ref{l1}, we denote $U=u/u_s$ and $V=v/v_s$. Averaging the above inequality in the $Y$ variable over a ball of small radius $1/\rho$ leads to
$$
\Big|U^a(X)-\fint_{B_{1/\rho}(X)}U^a(Y) dY\Big|\leq C\rho^{-a\alpha-\sigma}\fint_{B_{1/\rho}(X)}|X-Y|^\sigma dY=C\rho^{-a\alpha-2\sigma} \quad\mbox{ for all }\,  X\in B_{\rho/4}(z).
$$
We next let $\rho=|z|/4\to \infty$. Since $X\in B_{\rho/4}(z)$, this also implies  $|X|\geq |z|-\rho/4\to \infty$. Hence, for $\rho= |z|/2$ large enough the above estimate yields
$$
U^a(X)-\fint_{B_{1/\rho}(X)}U^a\to 0\quad\mbox{ as }\rho\to \infty.
$$
By \eqref{hy0} (applied to a translation of $U$), one has
$$
\fint_{B_{1/\rho}(X)}U^a\to 0\quad\mbox{ as }\rho\to \infty.
$$
Thus, $U^a(X)\to 0$ as $X\to \infty$. A similar conclusion holds for $V^b$ and we conclude.
\end{proof}
We can now apply Section \ref{secfast} to conclude that Theorems \ref{tbth} and \ref{tbth2} hold.

%\bigskip\noindent

\section*{Acknowledgement} The authors are greatful to D. Ye for his careful reading of a preliminary version of this paper and for pointing out Theorem 1.1 in \cite{ch}. H. Hajlaoui received funding from Claude Bernard University Lyon 1 (UCBL). Part of this work was done while he was visiting Lyon. He thanks the Department of Mathematics and the Faculty of Sciences for the kind hospitality and for the financial support.

\end{document}